\documentclass{article}
\usepackage{amsmath,amssymb,amsthm,mathrsfs}
\usepackage[colorlinks=true,linkcolor=blue]{hyperref}

\newtheorem{theorem}{Theorem}[section]
\newtheorem{proposition}{Proposition}[section]

\newtheorem{definition}{Definition}[section]
\newtheorem{remark}{Remark}[section]
\newtheorem{lemma}{Lemma}[section]
\newtheorem{example}{Example}[section]

\newcommand{\fonction}[5]{\begin{array}[t]{lrcl}#1 :&#2 &\longrightarrow &#3\\&#4& \longmapsto &#5 \end{array}}

\DeclareMathOperator*{\argmin}{arg\,min}
\DeclareMathOperator*{\argmax}{arg\,max}

\def\di{\displaystyle}
\def\a{\alpha}
\def\t{\tau}
\def\eps{\varepsilon}
\def\d{\mathrm{d}}

\def\R{\mathbb{R}}
\def\N{\mathbb{N}}
\def\L{\mathrm{L}}
\def\J{\mathrm{J}}
\def\E{\mathrm{E}}
\def\K{\mathrm{K}}

\def\PP{\mathrm{P}}
\def\PPP{\mathcal{P}}
\def\S{\mathrm{S}}
\def\LL{\mathcal{L}}
\def\C{\mathrm{C}}
\def\CCC{\mathcal{F}}
\def\B{\mathrm{B}}
\def\BB{\overline{\B}}
\def\AC{\mathrm{AC}}

\def\cc{\mathrm{c}}
\def\ACm{\AC^\a_{a+}}
\def\CACm{{}_{\cc} \ACm}
\def\CACmi{{}_{\cc} \AC^{\a,\infty}_{a+}}
\def\I{\mathrm{I}}
\def\Im{\I^\a_{a+}}
\def\Ip{\I^\a_{b-}}
\def\D{\mathrm{D}}
\def\CD{{}_{\cc} \mathrm{D}}
\def\Dm{\D^\a_{a+}}
\def\CDm{\CD^\a_{a+}}
\def\Dp{\D^\a_{b-}}
\def\CDp{\CD^\a_{b-}}

\title{Legendre's necessary condition for fractional Bolza functionals with mixed initial/final constraints}

\author{Lo\"ic Bourdin\footnote{Institut de Recherche XLIM. UMR CNRS 7252. Universit\'e de Limoges, France. \texttt{loic.bourdin@unilim.fr}}, Rui A.C. Ferreira\footnote{Grupo F\'isica-Matem\'atica, Faculdade de Ci\^encias, Universidade de Lisboa, Av. Prof. Gama Pinto, 2, 1649-003 Lisboa. \texttt{raferreira@fc.ul.pt}}}

\date{}

\begin{document}

\maketitle

\begin{abstract}
The present work was primarily motivated by our findings in the literature of some flaws within the proof of the second-order Legendre necessary optimality condition for fractional calculus of variations problems. Therefore we were eager to elaborate a correct proof and it turns out that this goal is highly nontrivial, especially when considering final constraints. This paper is the result of our reflections on this subject. 

Precisely we consider here a constrained minimization problem of a general Bolza functional that depends on a Caputo fractional derivative of order $0 < \alpha \leq 1$ and on a Riemann-Liouville fractional integral of order $\beta > 0$, the constraint set describing general mixed initial/final constraints. The main contribution of our work is to derive corresponding first- and second-order necessary optimality conditions, namely the Euler-Lagrange equation, the transversality conditions and, of course, the Legendre condition. A detailed discussion is provided on the obstructions encountered with the classical strategy, while the new proof that we propose here is based on the Ekeland variational principle.

Furthermore we underline that some subsidiary contributions are provided all along the paper. In particular we prove an independent and intrinsic  result of fractional calculus stating that it does not exist a nontrivial function which is, together with its Caputo fractional derivative of order~$0 < \alpha <1$, compactly supported. Moreover we also discuss some evidences claiming that Riemann-Liouville fractional integrals should be considered in the formulation of fractional calculus of variations problems in order to preserve the existence of solutions.
\end{abstract}


\medskip


\section{Introduction}
It is widely known that the \textit{calculus of variations} is nearly as old as the calculus, being both subjects developed in a somewhat parallel way (see, e.g.,~\cite[p.1]{Brunt2004}). It is commonly accepted that the calculus of variations was born with the brachistochrone problem posed by Bernouilli in 1696. The enduring interest in the calculus of variations is due to its numerous applications in several scientific areas, and in particular to its relationships with classical mechanics. Incidentally the \textit{fractional calculus} (or calculus using derivatives of any real order~$\alpha>0$) is also as old as the calculus, since fractional derivatives made part of Leibniz's studies (see, e.g.,~\cite[p.xxvii]{Samko}). Since a conference organized by Ross at the University of New Haven in 1974 dedicated to the applications of fractional calculus, this topic experiences a boom in several scientific fields. Its use is so wide that it seems difficult to give a complete overview of the current investigations involving fractional operators. We can at least mention that the fractional calculus is massively applied in the physical context of anomalous diffusion (see, e.g.,~\cite{Gerolymatou2006,Metzler2000,Oldham1970, Zaslavsky2002,Zaslavsky2008,Zoia2010,Zoia2009}). Due to the non-locality of the fractional operators, they are also used in order to take into account memory effects (see, e.g.,~\cite{Bagley1983,Bagley1986,Pfitzenreiter2004}) where viscoelasticity is modelled by a fractional differential equation. We also refer to studies in wave mechanic \cite{Almeida2010}, economy \cite{Comte2001}, biology \cite{Glockle1995,Magin2010}, acoustic \cite{Helie2006}, thermodynamic~\cite{Hilfer2000}, probability \cite{Levy1939}, etc. We refer to~\cite{Hilfer2000b,Sabatier2007} for a large panorama of applications of fractional calculus.

\medskip

Historically it seems that the first time when a fractional operator appeared in the resolution of a concrete problem goes back to the Abel's solution of the tautochrone problem in 1826. Since the tautochrone problem can be seen as a variant of the brachistochrone problem, one can see in history a kind of interconnection between the calculus of variations and the fractional calculus. However, to the best of our knowledge, these two mathematical fields have been formally linked together for the first time only in 1996 by Riewe~\cite{Riewe1996}, when looking for a variational structure of nonconservative systems. The fractional Euler-Lagrange equation was stated there for the first time in the literature. In the past twenty years, many papers containing novel related results have been produced (see, e.g., \cite{Agrawal,Atan,bastos,bourdintorres,bourdin2014,cresson,Cresson2012,ferreira2011,klimek,bale} and references therein). Unfortunately a great account of articles dealing with this recent subject possess many inconsistencies, and occasionally even present serious mistakes, as reported e.g. in~\cite{Ferreira2018,Ferreira2015}. Roughly speaking and in some sense, the \emph{fractional calculus of variations} is still in its infancy. The present work being focused on a general problem of this field, one of our objectives here is to provide a rigorous treatment, in particular by taking care of the functional framework considered (see Remarks~\ref{remunsuitable}, \ref{remsuitable} and~\ref{remcont} for example).

\paragraph{Initial motivation of the present work.}

The study that we are presenting here was primarily motivated by our findings of some flaws within the proof of the second-order Legendre necessary optimality condition for the most basic\footnote{As a first step towards the open challenge of proving a fractional version of the second-order Legendre necessary optimality condition, the cost functional considered in~\cite{Lazo2014} is the most basic one, containing (only) a classical Lagrange cost depending on a fractional derivative of order~$0 < \alpha \leq 1$, under fixed initial/final conditions.} fractional calculus of variations problem presented in 2014 in~\cite{Lazo2014}. Soon after we found another contemporary work~\cite{Guo2013} in which a proof of such condition was also presented (in fact for a more general fractional optimal control problem but without final constraint) with, again, some erroneous assertions. Incredibly, more recently, in a series of papers~\cite{Almeida2017,Almeida2018,AlmeidaMorgado}, the author(s) presented a completely analogous (and thus incorrect) proof to the one in~\cite{Lazo2014} for some variants of fractional variational problems. Therefore we became very eager to elaborate a correct proof of the Legendre condition in the fractional setting. It turns out that this goal is highly nontrivial, especially when considering final constraints. We refer to Section~\ref{secobstruction} for a detailed discussion on the obstructions encountered in that framework.

\paragraph{Contributions of the manuscript.}
Having in mind what we wrote in the two previous paragraphs, with this work we aim to provide a rigorous and thoughtful study of a general fractional calculus of variations problem. Precisely we consider a constrained minimization problem of a general Bolza functional that contains a Mayer cost, as well as a Lagrange cost written with a Riemann-Liouville fractional integral of order $\beta > 0$ and depending on a Caputo fractional derivative of order $0 < \alpha \leq 1$. Furthermore we emphasize that, in this paper, the constraint set will be described with general mixed initial/final constraints. We refer to the beginning of Section~\ref{OP2} for a detailed presentation of the problem considered here (see Problem~\eqref{minproblem}). The main contribution of the present manuscript is to derive corresponding first- and second-order necessary optimality conditions, namely the Euler-Lagrange equation, the transversality conditions and, of course, the Legendre condition. To emphasize the obstructions encountered in the fractional setting when considering final constraints, we provide two separate studies:
\begin{enumerate}
\item[\rm{(i)}] The first one concerns the (simple) case where no constraint is considered in Problem~\eqref{minproblem} (see Section~\ref{secwithout}). The main result of this part is given in Theorem~\ref{thmmain1}. Our proof is a nontrivial adaptation of the standard techniques from the classical calculus of variations to the fractional setting, in the sense that several technical difficulties arising from the nonlocality of the fractional operators have to be overcome (see the use of a technical inequality derived in Lemma~\ref{lem30} for the proof of Lemma~\ref{lem3} for example). 
\item[\rm{(ii)}] The second study is concerned with the (more difficult) case where general mixed initial/final constraints are considered in Problem~\eqref{minproblem} (see Section~\ref{secwith}). As already mentioned, note that a preliminary discussion on the emerged obstructions is firstly provided in Section~\ref{secobstruction}. Nevertheless we are able to prove our main result (see Theorem~\ref{thm2}) which encompasses Theorem~\ref{thmmain1} but with a completely different approach. In the classical case $\alpha = \beta = 1$, it is well-known that the necessary optimality conditions considered in this paper can be seen as consequences of the Pontryagin maximum principle (in short, PMP) from optimal control theory. We refer to~\cite{agrachev,bressan,bullo,hestenes,Liberzon,pontryagin,schattler} and references therein. Therefore our idea is to adapt a standard proof of the PMP to our particular fractional variational problem. In this paper we follow and adapt the proof based on the Ekeland variational principle~\cite{Ekeland1974}.
\end{enumerate}
The two above items fill our initial objective about the fractional version of the Legendre condition. Furthermore we precise that some subsidiary contributions are provided all along the paper:
\begin{enumerate}
\item[\rm{(iii)}] On one hand, in Section~\ref{secobstruction}, we will dwell into the details of the proofs presented in~\cite{Guo2013} and~\cite{Lazo2014} and enlighten the serious flaws contained within them.
\item[\rm{(iv)}] Along with our primary considerations about the obstructions encountered for the Legendre condition in the fractional setting with final constraints, a side result emerged that may be of independent interest for researchers in fractional calculus. The reader may find its contents in Proposition~\ref{prop765} in Section~\ref{secobstruction} (see also Proposition~\ref{prop765bis} in Appendix~\ref{app09}). Concretely this result asserts that it does not exist a nontrivial function which is, together with its Caputo fractional derivative of order~$0 < \alpha <1$, compactly supported. Note that this result is an intrinsic result of fractional calculus, in the sense that it is clearly not true for~$\alpha = 1$. We refer to Remark~\ref{remfracdiff} for some consequences in the fractional differential equations theory.
\item[\rm{(v)}] One particular feature of the present work is that we consider, in our formulation of the Bolza functional, a Riemann-Liouville fractional integral of order~$\beta>0$. At a first glance this seems to be only one more fractional generalization of the classical calculus of variations. In fact, as observed recently in~\cite[Section~4]{Ferreira2018}, there is evidence that Riemann-Liouville fractional integrals should indeed be considered in the formulation of fractional calculus of variations problems in order to guarantee existence of solutions. We shall continue in this line of thought and we provide here some more evidences supporting that claim (see Remark~\ref{rem465}). Precisely, as illustrated in Example~\ref{ex1}, a fractional generalization of a simple calculus of variations problem, without considering a Riemann-Liouville fractional integral (that is, by considering~$\beta=1$), may not preserve the existence of a solution. This issue turns out to be a direct consequence of the transversality conditions derived in Theorems~\ref{thmmain1} and~\ref{thm2}.
\end{enumerate}

\paragraph{Organization of the paper.} Section~\ref{OP1} is devoted to definitions and basic results from fractional calculus used throughout the text. In Section~\ref{OP2} we enunciate our main results (Theorems~\ref{thmmain1} and~\ref{thm2}) and provide a list of related comments. This section is divided in several subsections as explained in the above paragraph. Finally, for the reader's convenience, the quite long and technical proofs of Proposition~\ref{prop765} and Theorem~\ref{thm2} are given in Appendices~\ref{app09} and~\ref{secproof}, respectively. 

\section{Notations and basics from fractional calculus}\label{OP1}

Throughout the paper the abbreviation RL stands for Riemann-Liouville. This section is devoted to recall basic definitions and results about RL and Caputo fractional operators. All of the presented below is very standard and mostly extracted from the monographs~\cite{Kilbas,Samko}. We first introduce some functional framework. Let $n\ge 1$ be a fixed positive integer, $a < b$ be two fixed real numbers and~$1 \leq r \leq \infty$ be a fixed extended real number. In this paper we denote by:
\begin{itemize}
\item $\L^r := \L^r([a,b],\R^n)$ the Lebesgue space of $r$-integrable functions (or, if $r=\infty$, of essentially bounded functions) defined on $[a,b]$ with values in $\R^n$, endowed with its usual norm $\| \cdot \|_{\L^r}$;
\item $\C := \C([a,b],\R^n)$ the space of continuous functions defined on $[a,b]$ with values in $\R^n$, endowed with the uniform norm $\| \cdot \|_\C$;
\item $\AC := \AC([a,b],\R^n)$ the subspace of $\C$ of absolutely continuous functions;
\item $\C^r := \C^r([a,b],\R^n)$ the subspace of $\AC$ of $r$ times continuously differentiable functions (or, if $r=\infty$, of infinitely differentiable functions);
\item $\C^\infty_\cc := \C^\infty_\cc([a,b],\R^n)$ the subspace of $\C^\infty$ of infinitely differentiable functions with compact support included in $(a,b)$.
\end{itemize}
For any functional set $\E \subset \C$, we denote by $\E_0$ the set of all functions $x \in \E$ such that~$x(a) = 0_{\R^n}$. For example $\C^\infty_\cc \subset \C^\infty_0 \subset \AC_0 \subset \C_0 \subset \C$.

\subsection{Left RL and Caputo fractional operators}\label{secleft}
We start with left fractional integrals and derivatives of RL and Caputo types. In the sequel $\Gamma$ denotes the standard Gamma function.

\begin{definition}[Left RL fractional integral]
The left RL fractional integral $\Im [x]$ of order $\a > 0$ of a function $x \in \L^1$ is defined on $[a,b]$ by 
$$ \Im [x](t) :=  \int_a^t \dfrac{(t-s)^{\a-1}}{\Gamma(\a)} x(s) \, ds, $$
provided that the right-hand side term exists. For $\alpha = 0$ we set $\I^0_{a+}[x] := x$.
\end{definition}

\begin{proposition}[{\cite[Lemma 2.1 p.72]{Kilbas}}]\label{prop1}
If $\a \geq 0$ and $x \in \L^1$, then $\Im [x] \in \L^1$.
\end{proposition}

\begin{proposition}[{\cite[Lemma 2.3 p.73]{Kilbas}}]\label{prop5}
If $\a_1 \geq 0$, $\a_2 \geq 0$ and $x \in \L^1$, then the equalities
$$ \I^{\alpha_1}_{a+} \Big[ \I^{\alpha_2}_{a+} [ x ] \Big] = \I^{\alpha_1 +\alpha_2}_{a+} [ x ] = \I^{\alpha_2 +\alpha_1}_{a+} [ x ] = \I^{\alpha_2}_{a+} \Big[ \I^{\alpha_1}_{a+} [ x ] \Big], $$ 
hold true.
\end{proposition}

Let $\a \geq 0$ and $x \in \L^1$. From Proposition~\ref{prop1}, $\Im [x](t)$ exists for almost every $t \in [a,b]$. Throughout the paper, if $\Im [x]$ is equal almost everywhere on $[a,b]$ to a continuous function, then $\Im [x]$ is automatically identified to its continuous representative. In that case $\Im [x](t)$ is defined for every~$t \in [a,b]$.

\begin{proposition}[{\cite[Theorem 3.6 p.67]{Samko}}]\label{prop3}
If $\a > 0$ and $x \in \L^\infty$, then $\Im [x] \in \C_0$. 
\end{proposition}

\begin{definition}[Left RL fractional derivative]\label{defRLleft}
We say that $x \in \L^1$ possesses a left RL fractional derivative $\Dm [x]$ of order $0 \leq \a \leq 1$ if and only if $\I^{1-\a}_{a+} [x] \in \AC $. In that case $\Dm [x] \in \L^1$ is defined by
$$ \Dm [x](t) := \dfrac{d}{dt} \Big[ \I^{1-\a}_{a+} [x] \Big] (t) , $$
for almost every $t \in [a,b]$. We denote by $\AC^\alpha_{a+} := \AC^\alpha_{a+}([a,b],\R^n)$ the space of all functions~$x \in \L^1$ possessing a left RL fractional derivative $\Dm [x]$ of order $0 \leq \a \leq 1$.
\end{definition}
 
\begin{remark}
If $\a = 1$, $\AC^1_{a+} = \AC $ and $\D^1_{a+} [x] = \dot{x}$ for any $x \in \AC $. If $\a = 0$, $\AC^0_{a+}= \L^1 $ and $\D^0_{a+} [x] = x$ for any $x \in \L^1$.
\end{remark}

\begin{proposition}[{\cite[Proposition~5 p.220]{bourdinidczak}}]\label{propcompRL}
Let $0 \leq \a \leq 1$ and $x \in \L^1 $. Then $x \in \AC^\alpha_{a+}$ if and only if there exists~$(u,y) \in \L^1 \times \R^n$ such that
$$ x(t) = \dfrac{(t-a)^{\a-1}}{\Gamma(\a)}  y + \Im [u](t), $$
for almost every $t \in [a,b]$. In that case, it necessarily holds that $u = \Dm [x]$ and $y = \I^{1-\a}_{a+}[x](a)$.
\end{proposition}

\begin{remark}\label{remunsuitable}
In general a function $x \in \AC^\alpha_{a+}$ admits a singularity at $t=a$. As a consequence it might be unsuitable to define a Bolza functional on the functional space $\AC^\alpha_{a+}$. In order to avoid this pitfall, in this paper we will use the Caputo notion of fractional derivative, recalled below.
\end{remark}

\begin{definition}[Left Caputo fractional derivative]\label{defCaputoleft}
We say that $x \in \C$ possesses a left Caputo fractional derivative $\CDm [x]$ of order $0 \leq \a \leq 1$ if and only if $x-x(a) \in \AC^\a_{a+}$. In that case~$\CDm [x] \in \L^1$ is defined by
$$ \CDm [x](t) := \Dm [x-x(a)] (t), $$
for almost every $t \in [a,b]$. We denote by $ {}_\cc \AC^\alpha_{a+} := {}_\cc \AC^\alpha_{a+}([a,b],\R^n)$ the space of all functions $x \in \C$ possessing a left Caputo fractional derivative $\CDm [x]$ of order $0 \leq \a \leq 1$.
\end{definition}

\begin{remark}
If $\a = 1$, ${}_\cc \AC^1_{a+} = \AC$ and ${}_\cc \D^1_{a+} [x] = \dot{x}$ for any $x \in \AC$. If $\a = 0$, ${}_\cc \AC^0_{a+} = \C $ and ${}_\cc \D^0_{a+} [x] = x-x(a)$ for any $x \in \C$.
\end{remark}

\begin{remark}
Let $0 \leq \a \leq 1$. Note that ${}_\cc \AC^\alpha_{a+} = \AC^\a_{a+} \cap \C $ and, if~$\a \neq 1$, it holds that
$$ \CDm[x](t) = \Dm[x](t)-\dfrac{(t-a)^{-\a} }{\Gamma(1-\a)}x(a) ,$$
for almost every $t \in [a,b]$ and all $x \in {}_\cc \AC^\alpha_{a+}$.
\end{remark}

\begin{proposition}[{\cite[Proposition~2.13 p.6]{Bourdin2018}}]\label{propcompC}
Let $0 \leq \a \leq 1$ and $x \in \C$. Then $x \in {}_\cc \AC^\alpha_{a+}$ if and only if there exists~$(u,y) \in \L^1 \times \R^n$ such that
$$ x(t) = y + \Im [u](t), $$
for almost every $t \in I$. In that case, the above relation holds replacing $u$ by ${}_\cc \Dm [x]$ and $y$ by $x(a)$.
\end{proposition}

\begin{remark}[{\cite[Remark~2.14 p.6]{Bourdin2018}}]\label{rempropcompC2}
Let $0 < \a \leq 1$ and $ x(t) = y + \Im [u](t)$ for almost every $t \in [a,b]$, for some $(u,y) \in \L^1 \times \R^n$. It might be possible that $x \notin \C$ and then Proposition~\ref{propcompC} cannot be applied. From Proposition~\ref{prop3}, if moreover $u \in \L^\infty$, then $x \in \C$ with $x(a)=y$ and Proposition~\ref{propcompC} can be applied. In that situation one can also conclude that $u = {}_\cc \Dm [x]$.
\end{remark}

\begin{remark}[{\cite[Remark~2.15 p.6]{Bourdin2018}}]\label{rempropcompC}
Let $0 < \a \leq 1$ and $x \in \C$ such that $ x(t) = y + \Im [u](t)$ for almost every $t \in [a,b]$, for some $(u,y) \in \L^1 \times \R^n$. From Proposition~\ref{propcompC}, we know that $x \in {}_\cc \AC^\alpha_{a+}$ and that $ x(t) = x(a) + \Im [ {}_\cc \Dm [x] ](t)$ for almost every $t \in [a,b]$. However, without any additional assumption, one cannot assert that $u={}_\cc \Dm [x]$ and $y = x(a)$. From Proposition~\ref{prop3}, if moreover $u \in \L^\infty $, then one can conclude that $u={}_\cc \Dm [x]$ and $y = x(a)$.
\end{remark}

From the above definitions and propositions, one can easily recover the following well-known result.

\begin{proposition}[{\cite[Theorem~2.1 p.92]{Kilbas}}]\label{prop5bis}
Let $0 \leq \a \leq 1$. The inclusion $\AC \subset {}_\cc \AC^\alpha_{a+}$ holds true with $ \CDm [x] = \I^{1-\a}_{a+} [ \dot{x} ] $ for any $x \in \AC$.
\end{proposition}

\subsection{Right RL and Caputo fractional operators}

This section is devoted to the definitions of right fractional integrals and derivatives of RL and Caputo types.

\begin{definition}[Right RL fractional integral]\label{defRLright}
The right RL fractional integral $\Ip [x]$ of order $\a > 0$ of $x \in \L^1$ is defined on~$[a,b]$ by
$$ \Ip [x](t) := \int_t^b  \dfrac{(s-t)^{\a-1}}{\Gamma(\a)} x(s) \, ds, $$
provided that the right-hand side term exists. For $\alpha = 0$ we define $\I^0_{b-}[x] := x$.
\end{definition}

\begin{definition}[Right RL fractional derivative]
We say that $x \in \L^1$ possesses a right RL fractional derivative $\Dp [x]$ of order $0 \leq \a \leq 1$ if and only if $\I^{1-\a}_{b-} [x] \in \AC $. In that case~$\Dp [x] \in \L^1$ is defined by
$$ \Dp [x](t) := -\dfrac{d}{dt} \Big[ \I^{1-\a}_{b-} [x] \Big] (t) ,$$
for almost every $t \in [a,b]$. We denote by $\AC^\alpha_{b-} := \AC^\alpha_{b-}([a,b],\R^n)$ the set of all functions $x \in \L^1$ possessing a right RL fractional derivative $\Dp [x]$ of order $0 \leq \a \leq 1$.
\end{definition}

\begin{definition}[Right Caputo fractional derivative]\label{defCaputoright}
We say that $x \in \C$ possesses a right Caputo fractional derivative $\CDp [x]$ of order $0 \leq \a \leq 1$ if and only if $x-x(b) \in \AC^\a_{b-} $. In that case~$\CDp [x] \in \L^1$ is defined by
$$ \CDp [x](t) := \Dp [x-x(b)] (t) ,$$
for almost every $t \in [a,b]$. We denote by ${}_\cc \AC^\alpha_{b-} := {}_\cc \AC^\alpha_{b-}([a,b],\R^n)$ the set of all functions $x \in \C$ possessing a right Caputo fractional derivative $\CDp [x]$ of order $0 \leq \a \leq 1$.
\end{definition}

Each result stated in Section~\ref{secleft} (for left fractional operators) has a right-counterpart version. We refer the reader to~\cite{Kilbas,Samko} for details.

\section{Main results}\label{OP2}
In this work we focus on the constrained minimization problem
\begin{equation}\label{minproblem}\tag{P}
\argmin_{x \in \K} \; \LL(x),
\end{equation}
where $\LL$ is the general fractional functional of Bolza form given by
$$ \fonction{\LL}{{}_\cc \AC^{\a,\infty}_{a+}}{\R}{x}{\LL (x) := \varphi (x(a),x(b)) + \I^{\beta}_{a+} \left[ L(x,\CDm [x],\cdot) \right](b) ,} $$
where $0 < \a \leq 1$ and $\beta > 0$, where $\varphi : \R^n \times \R^n \to \R$ and $L : \R^n \times \R^n \times [a,b] \to \R$ are smooth of class~$\C^2$, where
$$ {}_\cc \AC^{\a,\infty}_{a+} := \{ x \in \CACm \mid \CDm [x] \in \L^\infty \}, $$
and where $\K \subset {}_\cc \AC^{\a,\infty}_{a+}$ stands for a general set of constraints. 

\begin{remark}
Consider the above definition of~$\LL(x)$ for some~$x \in {}_\cc \AC^{\a,\infty}_{a+}$. In this paper, in order to avoid heavy notations, we decided to use the notation~$\I^{\beta}_{a+} [ L(x,\CDm [x],\cdot) ](b)$ to stand for the value at~$t=b$ of the left~RL fractional integral of order~$\beta$ of the function~$t \mapsto L(x(t),\CDm [x](t),t)$. This notation has been preferred to the (heavier) notation~$\I^{\beta}_{a+} [ L(x(\cdot),\CDm [x](\cdot),\cdot) ](b)$, which can be found in the literature, in the purpose of lightening the notations used in the following statements, proofs and computations.
\end{remark}

\begin{remark}\label{remsuitable}
As opposed to $\AC^{\a}_{a+}$ (see Remark~\ref{remunsuitable}), note that the functional space $ {}_\cc \AC^{\a,\infty}_{a+}$ is suitable in order to define correctly the Bolza functional~$\LL$. Indeed, for all $x \in {}_\cc \AC^{\a,\infty}_{a+}$, the Mayer term $\varphi (x(a),x(b))$ is well-defined because $x \in \C$, and the Lagrange term $\I^{\beta}_{a+} [ L(x,\CDm [x],\cdot) ](b)$ is well-defined since $\I^{\beta}_{a+} [ L(x,\CDm [x],\cdot) ] \in \C$ because $L(x,\CDm [x],\cdot) \in \L^\infty$ (see Proposition~\ref{prop3}).
\end{remark}

\begin{remark}
Note that $\C^1 \subset \CACmi$ from Propositions~\ref{prop3} and~\ref{prop5bis}, while $\AC$ does not.
\end{remark}

Our aim in this paper is to derive first- and second-order necessary optimality conditions for Problem~\eqref{minproblem}. To this aim we first compute in Section~\ref{secgateaux} the first- and second-order G\^ateaux-differentials of~$\LL$. In Section~\ref{secwithout} we deal with the (easy) case where there is no constraint in Problem~\eqref{minproblem}, that is, the case $\K = {}_\cc \AC^{\a,\infty}_{a+}$ (see Theorem~\ref{thmmain1}). In Section~\ref{secobstruction} we discuss some obstructions arising in the fractional setting when considering final constraints in Problem~\eqref{minproblem} (such as fixed endpoint). In Section~\ref{secwith} we deal with the more general framework in which general mixed initial/final constraints are considered in Problem~\eqref{minproblem} (see Theorem~\ref{thm2}).

\medskip

Before coming to these points, we will use in the sequel the Lebesgue Dominated Convergence theorem at several occasions. It will be abbreviated by LDC theorem. Similarly, the Partial Converse of the Lebesgue Dominated Convergence theorem will be denoted by PCLDC theorem. 

\subsection{First- and second-order G\^ateaux-differentials of the Bolza functional}\label{secgateaux}
In this section we are concerned with the first- and second-order G\^ateaux-differentials of the Bolza functional~$\LL$.

\begin{definition}
The (first-order) \textit{G\^ateaux-differential} of~$\LL$ at some~$x \in {}_\cc \AC^{\a,\infty}_{a+}$ is defined by
$$ \mathcal{D}\LL (x)(\eta) := \lim\limits_{h \to 0} \dfrac{\LL(x+h\eta)-\LL(x)}{h}, $$
for all $\eta \in {}_\cc \AC^{\a,\infty}_{a+}$, provided that the right-hand side term exists. 
\end{definition}

\begin{proposition}\label{propfirstdifferential}
It holds that
\begin{multline}\label{eqfirstdifferential}
\mathcal{D}\LL (x)(\eta) = \partial_1 \varphi (x(a),x(b)) \cdot \eta (a) + \partial_2 \varphi (x(a),x(b)) \cdot \eta (b) \\[5pt]
+ \I^{\beta}_{a+} \Big[ \partial_1 L(x,\CDm [x],\cdot) \cdot \eta + \partial_2 L(x,\CDm [x],\cdot) \cdot \CDm [\eta] \Big](b),
\end{multline}
for all $x$, $\eta \in {}_\cc \AC^{\a,\infty}_{a+}$. 
\end{proposition}

\begin{proof}
This result can be derived from a first-order Taylor expansion with integral rest of $L$ and from the LDC theorem.
\end{proof}

\begin{definition}
The \textit{second-order G\^ateaux-differential} of~$\LL$ at some~$x \in {}_\cc \AC^{\a,\infty}_{a+}$ is defined by
$$ \mathcal{D}^2\LL (x)(\eta) := \lim\limits_{h \to 0} \dfrac{\LL(x+h\eta)-\LL(x) - h \mathcal{D}\LL (x) ( \eta)}{ h^2 / 2 }, $$
for all $\eta \in {}_\cc \AC^{\a,\infty}_{a+}$, provided that the right-hand side term exists.
\end{definition}

\begin{proposition}
It holds that
\begin{multline}\label{eqseconddifferential}
\mathcal{D}\LL^2 (x)(\eta) = \eta (a)^\top \times A \times \eta (a) + 2 \eta (a)^\top \times B \times \eta (b)+ \eta(b)^\top \times C \times \eta (b) \\[5pt]
+ \I^{\beta}_{a+} \Big[ \eta^\top \times P \times \eta + 2 \eta^\top \times Q \times \CDm[\eta] + \CDm[\eta]^\top \times R \times \CDm[\eta] \Big](b),
\end{multline}
for all $x$, $\eta \in {}_\cc \AC^{\a,\infty}_{a+}$, where
\begin{eqnarray*}
A := \partial^2_{11} \varphi (x(a),x(b)), & B := \partial^2_{12} \varphi (x(a),x(b)), & C := \partial^2_{22} \varphi (x(a),x(b)), \\
P := \partial^2_{11} L (x, \CDm [x],\cdot), & Q := \partial^2_{12} L (x, \CDm [x],\cdot), & R := \partial^2_{22} L (x, \CDm [x],\cdot).
\end{eqnarray*}
\end{proposition}

\begin{proof}
This result can be proved using a second-order Taylor expansion with integral rest of~$L$ and the LDC theorem.
\end{proof}

Equalities~\eqref{eqfirstdifferential} and~\eqref{eqseconddifferential} are required in order to prove our main result in Section~\ref{secwithout} (see the proof of Theorem~\ref{thmmain1}).

\subsection{The case without constraints}\label{secwithout}
This section is dedicated to the (easy) case where there is no constraint in Problem~\eqref{minproblem}, that is, the case $\K = {}_\cc \AC^{\a,\infty}_{a+}$. The main result of this section (see Theorem~\ref{thmmain1} below) is based on the following series of lemmas. Note that the first lemma (Lemma~\ref{lem1}) is a well-known result that can be found in the literature (see, e.g., \cite[Lemma 15.1 p.50]{hestenes}) under different presentations and/or under different names (such as \textit{Du Bois-Reymond lemma} or \textit{fundamental lemma of the calculus of variations}). The short proof of Lemma~\ref{lem1} is recalled here for the reader's convenience.

\begin{lemma}\label{lem1}
Let $x_1$, $x_2 \in \L^1$. If
$$ \I^1_{a+} \Big[ x_1 \cdot \eta +  x_2 \cdot \dot{\eta} \Big] (b) = 0,  $$
for all $\eta \in \C^\infty_\cc$, then $x_2 \in \AC$ (precisely, it admits an absolutely continuous representative) with~$\dot{x_2}(t)=x_1 (t)$ for almost every $t \in [a,b]$.
\end{lemma}

\begin{proof}
A classical integration by parts formula leads to
$$ \I^1_{a+} \Big[ x_1 \cdot \eta +  x_2 \cdot \dot{\eta} \Big] (b) = \di \int_a^b x_1(s) \cdot \eta (s) + x_2 (s) \cdot \dot{\eta}(s) \, ds =  \di \int_a^b ( x_2 (s) - X_1(s) ) \cdot \dot{\eta} (s) \, ds = 0, $$
for all $\eta \in \C^\infty_\cc$, where $X_1 \in \AC$ is defined by $X_1(t) := \int_a^t x_1 (s) \, ds$ for all $t \in [a,b]$. From the classical distribution theory, we deduce that there exists a constant $y \in \R^n$ such that $x_2 (t) - X_1(t) = y$ for almost every $t \in [a,b]$ which concludes the proof.
\end{proof}

\begin{lemma}\label{lem2}
Let $0 < \a \leq 1$ and $\beta > 0$. Then:
\begin{enumerate}
\item[\rm{(i)}] The equality
$$ \I^{\beta}_{a+} [x](b) = \I^{\gamma}_{a+} \left[ \dfrac{\Gamma(\gamma)}{\Gamma(\beta)} (b-\cdot)^{\beta - \gamma} x \right](b), $$
is satisfied for all $x \in \L^\infty$ and all $0 < \gamma < \beta + 1$ (in particular for $\gamma = 1$);
\item[\rm{(ii)}] The fractional integration by parts
$$ \I^{\beta}_{a+} \left[ \I^{1-\a}_{a+} [x_1] \cdot x_2 \right](b) =  \I^{1}_{a+} \left[  x_1 \cdot \I^{1-\a}_{b-} \left[ \frac{(b-\cdot)^{\beta-1}}{\Gamma(\beta)}  x_2 \right] \right](b),  $$
holds true for all $x_1$, $x_2 \in \L^\infty$.
\end{enumerate}
\end{lemma}

\begin{proof}
The first item is obvious. The second item is a simple consequence of the classical Fubini theorem.
\end{proof}

\begin{lemma}\label{lem30}
Let $0 < \a \leq 1$. The inequality
\begin{equation*}
0 \leq ( s_2^\alpha - s_1^\alpha )^2 \leq \alpha (s_2-s_1)^{\alpha+1} s_1^{\alpha-1},
\end{equation*}
holds true for all $0 < s_1 \leq s_2$.
\end{lemma}

\begin{proof}
Lemma~\ref{lem30} directly follows from the inequalities
\begin{equation*}
0 \leq s_2^\alpha - s_1^\alpha \leq (s_2-s_1)^\alpha \quad \text{and} \quad 0 \leq s_2^\alpha - s_1^\alpha \leq \alpha (s_2-s_1) s_1^{\alpha-1},
\end{equation*}
both holding for all $0 < s_1 \leq s_2$. The first inequality can be easily obtained by studying the real function $f(s) := (s-1)^{\a} -s^\a + 1$ for $s \geq 1$. The second inequality can be obtained by the mean value theorem.
\end{proof}

\begin{lemma}\label{lem3}
Let $0 < \a \leq 1$ and $\beta > 0$. Let $A$, $B$, $C \in \R^{n \times n}$ be three matrices and let $P$, $Q$, $R \in \L^\infty([a,b],\R^{n \times n})$ be three essentially bounded matrix functions. If it holds that
\begin{multline*}
\eta (a)^\top \times A \times \eta (a) + 2 \eta (a)^\top \times B \times \eta (b)+ \eta(b)^\top \times C \times \eta (b) \\[5pt]
+ \I^{\beta}_{a+} \Big[ \eta^\top \times P \times \eta + 2 \eta^\top \times Q \times \CDm[\eta] + \CDm[\eta]^\top \times R \times \CDm[\eta] \Big](b) \geq 0,
\end{multline*}
for all $\eta \in {}_\cc \AC^{\a,\infty}_{a+}$, then $\frac{(b-t)^{\beta-1}}{\Gamma(\beta)} R(t)$ is positive semi-definite for almost every $t \in [a,b]$.
\end{lemma}

\begin{proof}
Let us consider a Lebesgue point $\t \in (a,b)$ of the matrix function $\frac{(b-\cdot)^{\beta-1}}{\Gamma(\beta)} R \in \L^1([a,b],\R^{n \times n})$ and let~$v \in \R^n$. Our aim is to prove that $ \frac{(b-\t)^{\beta-1}}{\Gamma(\beta)} v^\top \times R(\t) \times v \geq 0$. To this aim we consider~$\eta := \Im [ \nu ]$ where~$\nu \in \L^\infty$ is defined by
$$ \nu (s) := \left\lbrace 
\begin{array}{lcl}
0_{\R^n} & \text{on} & [a,\t), \\
v & \text{on} & [\t,\t+h), \\
0_{\R^n} & \text{on} & [\t+h,b], 
\end{array}
\right. $$
for all $s \in [a,b]$ and for some small $0 < h \leq b-\t$. One can easily see that $\eta \in {}_\cc \AC^{\a,\infty}_{a+}$ with~$\CDm [\eta] = \nu$. Moreover it can be computed that
\begin{equation}\label{eqguo}
\eta (t)  := \left\lbrace 
\begin{array}{lcl}
0_{\R^n} & \text{on} & [a,\t], \\[8pt]
\dfrac{1}{\Gamma(1+\a)} (t-\t)^\a v & \text{on} & [\t,\t+h], \\[8pt]
\dfrac{1}{\Gamma(1+\a)} \Big( (t-\t)^\a - (t-(\t+h))^\a \Big) v & \text{on} & [\t+h,b] ,
\end{array} \right.
\end{equation}
for all $t \in [a,b]$.  We get from the inequality hypothesis that
\begin{multline*}
\dfrac{1}{h} \left( \eta(b)^\top \times C \times \eta (b) + \di \int_{\t}^{\t+h} \dfrac{(b-s)^{\beta-1}}{\Gamma(\beta)} v^\top \times R(s) \times v \, ds  \right. \\[5pt]
+ \int_{\t}^{\t+h} \dfrac{(b-s)^{\beta-1}}{\Gamma(\beta)} 2 \eta(s)^\top \times Q(s) \times v \, ds  \\[5pt]
\left. +  \int_{\t}^{\t+h} \dfrac{(b-s)^{\beta-1}}{\Gamma(\beta)} \eta(s)^\top \times P(s) \times \eta(s) \, ds + \int_{\t+h}^b \dfrac{(b-s)^{\beta-1}}{\Gamma(\beta)} \eta(s)^\top \times P(s) \times \eta(s) \, ds \right) \geq 0.
\end{multline*}
In the sequel we are concerned with the limit of the five above terms when $h \to 0^+$. From Lemma~\ref{lem30}, the first term can be bounded as follows
\begin{multline*}
\left\vert \dfrac{1}{h} \eta(b)^\top \times C \times \eta (b) \right\vert \leq \dfrac{\Vert C \Vert_{\R^{n \times n}} \Vert v \Vert^2_{\R^n} }{\Gamma(1+\a)^2} \dfrac{1}{h} \Big( (b-\t)^\a - (b-(\t+h))^{\a} \Big)^2 \\[5pt] 
 \leq   \dfrac{\Vert C \Vert_{\R^{n \times n}} \Vert v \Vert^2_{\R^n}}{\Gamma(1+\a)^2} \alpha h^\alpha (b-(\t+h))^{\a-1} , 
\end{multline*}
which tends to zero when $h \to 0^+$. Since $\t$ is a Lebesgue point of $\frac{(b-\cdot)^{\beta-1}}{\Gamma(\beta)} R$, the second limit
$$ \lim\limits_{h \to 0^+} \dfrac{1}{h} \di \int_{\t}^{\t+h} \dfrac{(b-s)^{\beta-1}}{\Gamma(\beta)} v^\top \times R(s) \times v \, ds = \dfrac{(b-\t)^{\beta-1}}{\Gamma(\beta)} v^\top \times R(\t) \times v, $$
holds true. Denoting by $M_h := \max (  (b-(\t+h))^{\beta-1},(b-\t)^{\beta-1} )$ (depending on $\beta < 1$ or $\beta \geq 1$), the third term can be bounded as follows
\begin{multline*}
\left\vert \dfrac{1}{h} \di \int_{\t}^{\t+h} \dfrac{(b-s)^{\beta-1}}{\Gamma(\beta)} 2  \eta(s)^\top \times Q(s) \times v \, ds  \right\vert \\[5pt]
 \leq  \dfrac{2 \Vert Q \Vert_{\L^\infty} \Vert v \Vert^2_{\R^n} M_h}{\Gamma(\beta) \Gamma(1+\a)} \dfrac{1}{h} \int_\t^{\t+h} (s-\t)^{\a} \, ds
= \dfrac{2 \Vert Q \Vert_{\L^\infty} \Vert v \Vert^2_{\R^n} M_h}{\Gamma(\beta) \Gamma(2+\a)} h^\alpha,
\end{multline*}
which tends to zero when $h \to 0^+$. The fourth term can be bounded as follows
\begin{multline*}
 \left\vert \dfrac{1}{h} \int_{\t}^{\t+h} \dfrac{(b-s)^{\beta-1}}{\Gamma(\beta)} \eta(s)^\top \times P(s) \times \eta(s) \, ds \right\vert \\[5pt]
 \leq  \dfrac{\Vert P \Vert_{\L^\infty} \Vert v \Vert^2_{\R^n} M_h}{\Gamma(\beta) \Gamma(1+\a)^2} \dfrac{1}{h} \int_\t^{\t+h} (s-\t)^{2\a} \, ds =  \dfrac{\Vert P \Vert_{\L^\infty} \Vert v \Vert^2_{\R^n} M_h}{\Gamma(\beta) (2\a+1) \Gamma(1+\a)^2}h^{2\alpha},
 \end{multline*}
which tends to zero when $h \to 0^+$. Finally, using Lemma~\ref{lem30}, the fifth term can be bounded as follows
\begin{multline*}
\left\vert \dfrac{1}{h} \int_{\t+h}^b \dfrac{(b-s)^{\beta-1}}{\Gamma(\beta)} \eta(s)^\top \times P(s) \times \eta(s) \, ds \right\vert \\[5pt]
\leq \dfrac{ \Vert P \Vert_{\L^\infty} \Vert v \Vert^2_{\R^n}}{\Gamma (\beta) \Gamma(1+\a)^2} \dfrac{1}{h} \int_{\t+h}^{b} (b-s)^{\beta-1} \Big( (s-\t)^\a - (s-(\t+h))^\a \Big)^2 \, ds \\[5pt]
 \leq \dfrac{ \Vert P \Vert_{\L^\infty} \Vert v \Vert^2_{\R^n} }{\Gamma (\beta) \Gamma(1+\a)^2} \alpha h^\alpha \int_{\t+h}^{b} (b-s)^{\beta-1} (s-(\t+h))^{\a -1} \, ds \\[5pt]
= \dfrac{ \Vert P \Vert_{\L^\infty} \Vert v \Vert^2_{\R^n} }{ \Gamma (\alpha+\beta) \Gamma(1+\a)} h^\alpha (b-(\t+h))^{\alpha+\beta-1}, 
\end{multline*}
which tends to zero when $h \to 0^+$. Thus we have obtained that $ \frac{(b-\t)^{\beta-1}}{\Gamma(\beta)} v^\top \times R(\t) \times v \geq 0$ and the proof is complete.
\end{proof}

We are now in a position to state and prove the main result of this section.

\begin{theorem}\label{thmmain1}
Let us assume that there is no constraint in Problem~\eqref{minproblem}, that is, $\K = {}_\cc \AC^{\a,\infty}_{a+}$. If $x \in {}_\cc \AC^{\a,\infty}_{a+}$ is a solution to Problem~\eqref{minproblem}, then:
\begin{enumerate}
\item[\rm{(i)}] \textbf{Euler-Lagrange equation}: the function $\frac{(b-\cdot)^{\beta-1}}{\Gamma (\beta)} \partial_2 L(x,\CDm [x],\cdot) \in \AC^\a_{b+}$ with
\begin{equation*}\label{eqeuler}
\Dp \left[ \dfrac{(b-\cdot)^{\beta-1}}{\Gamma (\beta)} \partial_2 L(x,\CDm [x],\cdot)  \right](t) = - \dfrac{(b-t)^{\beta-1}}{\Gamma (\beta)} \partial_1 L(x(t),\CDm [x](t),t),
\end{equation*}
for almost every $t \in [a,b]$;
\item[\rm{(ii)}] \textbf{Transversality conditions}: the equalities
\begin{eqnarray}
\I^{1-\a}_{b-}  \left[ \dfrac{(b-\cdot)^{\beta-1}}{\Gamma (\beta)} \partial_2 L(x,\CDm [x],\cdot)  \right] (a) & = &  \partial_1 \varphi ( x(a) , x(b)),\label{eqtransv1} \\
 - \I^{1-\a}_{b-}  \left[ \dfrac{(b-\cdot)^{\beta-1}}{\Gamma (\beta)} \partial_2 L(x,\CDm [x],\cdot)  \right] (b) & = & \partial_2 \varphi ( x(a) , x(b)), \label{eqtransv2}
\end{eqnarray}
are both satisfied;
\item[\rm{(iii)}] \textbf{Legendre condition}: the matrix~$\frac{(b-t)^{\beta-1}}{\Gamma(\beta)} \partial^2_{22} L(x(t),\CDm [x](t),t) \in \R^{n \times n}$ is positive semi-definite for almost every $t \in [a,b]$.
\end{enumerate} 
\end{theorem}

\begin{proof}
Since there is no constraint in Problem~\eqref{minproblem} and since $x \in {}_\cc \AC^{\a,\infty}_{a+}$ is a solution to Problem~\eqref{minproblem}, one can easily see that~$\mathcal{D}\LL(x)(\eta) = 0 $ and $\mathcal{D}^2 \LL(x)(\eta) \geq 0 $ for all $\eta \in {}_\cc \AC^{\a,\infty}_{a+}$. In particular, from Equality~\eqref{eqfirstdifferential}, it holds that 
$$  \I^{\beta}_{a+} \Big[ \partial_1 L(x,\CDm [x],\cdot) \cdot \eta + \partial_2 L(x,\CDm [x],\cdot) \cdot  \I^{1-\alpha}_{a+} [\dot{\eta}] \Big](b) = 0, $$
for all $\eta \in \C^\infty_\cc \subset  {}_\cc \AC^{\a,\infty}_{a+}$, since $\eta(a)=\eta(b)=0_{\R^n}$ and $\CDm [\eta] = \I^{1-\alpha}_{a+} [\dot{\eta}]$. Using the first item of Lemma~\ref{lem2} with $\gamma = 1$ on the first above term and the second item on the second above term, we get that 
$$ \I^{1}_{a+} \left[ \dfrac{(b-\cdot)^{\beta-1}}{\Gamma(\beta)} \partial_1 L(x,\CDm [x],\cdot) \cdot \eta + \I^{1-\alpha}_{b-} \left[ \dfrac{(b-\cdot)^{\beta-1}}{\Gamma(\beta)} \partial_2 L(x,\CDm [x],\cdot) \right] \cdot \dot{\eta} \right](b) = 0, $$
for all $\eta \in \C^\infty_\cc$. From Lemma~\ref{lem1}, we conclude that 
$$ \I^{1-\alpha}_{b-} \left[ \dfrac{(b-\cdot)^{\beta-1}}{\Gamma(\beta)} \partial_2 L(x,\CDm [x],\cdot) \right] \in \AC ,$$
which corresponds exactly to $\frac{(b-\cdot)^{\beta-1}}{\Gamma (\beta)} \partial_2 L(x,\CDm [x],\cdot) \in \AC^\a_{b+}$, with
\begin{equation*}
\frac{d}{dt} \left[ \I^{1-\alpha}_{b-} \left[ \dfrac{(b-\cdot)^{\beta-1}}{\Gamma(\beta)} \partial_2 L(x,\CDm [x],\cdot) \right] \right](t) = \dfrac{(b-t)^{\beta-1}}{\Gamma (\beta)} \partial_1 L(x(t),\CDm [x](t),t),
\end{equation*}
for almost every $t \in [a,b]$, which exactly coincides with the Euler-Lagrange equation stated in Theorem~\ref{thmmain1}. In order to derive the transversality conditions, we follow the same strategy but with variations $\eta$ in the larger space $\C^\infty$ and we get that
\begin{multline*}
\partial_1 \varphi (x(a),x(b)) \cdot \eta (a) + \partial_2 \varphi (x(a),x(b)) \cdot \eta (b) \\[5pt]
+ \I^{1}_{a+} \left[ \dfrac{(b-\cdot)^{\beta-1}}{\Gamma(\beta)} \partial_1 L(x,\CDm [x],\cdot) \cdot \eta + \I^{1-\alpha}_{b-} \left[ \dfrac{(b-\cdot)^{\beta-1}}{\Gamma(\beta)} \partial_2 L(x,\CDm [x],\cdot) \right] \cdot \dot{\eta} \right](b) = 0,
\end{multline*}
for all $\eta \in \C^\infty$. From the Euler-Lagrange equation and with a simple integration, we get that
\begin{multline*}
\left( \partial_1 \varphi (x(a),x(b)) - \I^{1-\alpha}_{b-} \left[ \dfrac{(b-\cdot)^{\beta-1}}{\Gamma(\beta)} \partial_2 L(x,\CDm [x],\cdot) \right](a) \right) \cdot \eta (a) \\[5pt]
+ \left( \partial_2 \varphi (x(a),x(b)) + \I^{1-\alpha}_{b-} \left[ \dfrac{(b-\cdot)^{\beta-1}}{\Gamma(\beta)} \partial_2 L(x,\CDm [x],\cdot) \right](b) \right) \cdot \eta (b)  = 0,
\end{multline*}
for all $\eta \in \C^\infty$, which concludes the proof of the transversality conditions. Finally, the Legendre condition is a direct consequence of Equality~\eqref{eqseconddifferential} and Lemma~\ref{lem3}.
\end{proof}

A list of comments is in order.

\begin{remark}
In the literature, the Euler-Lagrange equation and the transversality conditions are known to be \textit{first-order necessary optimality conditions} (because they are derived from the vanishing of the first-order G\^ateaux-differential of $\LL$ at $x$). In constrast, the Legendre condition is known to be a \textit{second-order necessary optimality condition} (because it is derived from the nonnegativeness of the second-order G\^ateaux-differential of $\LL$ at $x$).
\end{remark}

\begin{remark}
In order to derive the (second-order) Legendre condition in Theorem~\ref{thmmain1}, we assumed in this paper that $\varphi$ and $L$ are smooth of class $\C^2$. Actually, note that the derivation of the first-order necessary optimality conditions in Theorem~\ref{thmmain1} only requires that $\varphi$ and $L$ are smooth of class $\C^1$. 
\end{remark}

\begin{remark}
For the sake of simplicity, we choose to enunciate Theorem~\ref{thmmain1} for (global) minimizers of the Bolza functional $\LL$. Nevertheless one can easily see that Theorem~\ref{thmmain1} can also be derived for local minimizers (in a sense to precise). Actually the first-order necessary optimality conditions in Theorem~\ref{thmmain1} are even valid for (only) critical points of the Bolza functional $\LL$.
\end{remark}

\begin{remark}\label{remcondfixee}
The proof of the first-order necessary optimality conditions in Theorem~\ref{thmmain1} can easily be adapted to the three following constrained cases:
\begin{enumerate}
\item[\rm{(i)}] fixed initial condition to some $x_a \in \R^n$ and free final condition, that is, if $\K = \{ x \in \CACmi \mid x(a) = x_a \}$; however, in that case, the transversality condition~\eqref{eqtransv1} cannot be derived in general;
\item[\rm{(ii)}] free initial condition and fixed final condition to some $x_b \in \R^n$, that is, if $\K = \{ x \in \CACmi \mid x(b) = x_b \}$; however, in that case, the transversality condition~\eqref{eqtransv2} cannot be derived in general;
\item[\rm{(iii)}] fixed initial condition to some $x_a \in \R^n$ and fixed final condition to some $x_b \in \R^n$, that is, if $\K = \{ x \in \CACmi \mid x(a) = x_a \text{ and } x(b) = x_b \}$; however, in that case, none of the transversality conditions is valid in general.
\end{enumerate}
On the contrary, the proof of the (second-order) Legendre condition can only be adapted to the first above constrained case. In Section~\ref{secobstruction} a detailed discussion is given about the obstructions encountered in deriving the Legendre condition in the fractional setting with final constraints (such as in the two last above constrained cases). In Section~\ref{secwith} we elaborate a new strategy in order to prove all the necessary optimality conditions of Theorem~\ref{thmmain1} (including the Legendre condition) in a general framework of mixed initial/final constraints (encompassing in particular  all the three above constrained cases).
\end{remark}

\begin{remark}\label{remcont}
This remark is devoted to the related problem of minimizing the restriction~$\LL_{|\CCC} : \CCC \to \R$ where $ \CCC := \{ x \in \CACm \mid \CDm[x] \in \C \}$. Note that $\C^1 \subset \CCC \subset \CACmi$ from Propositions~\ref{prop3} and~\ref{prop5bis}. Let $x \in \CCC$ be a minimizer of~$\LL_{|\CCC}$. Our aim is to prove that $x$ is then a minimizer of~$\LL$. Let $X \in \CACmi$, let us denote by~$u := \CDm[X] \in \L^\infty$ and consider some~$p > \frac{1}{\alpha}$ being fixed. From a standard density result (see, e.g.,~\cite[Corollary~4.23 p.109]{brezis}), there exists a sequence $(u_k)_{k \in \N} \subset \C^\infty_\cc$ such that $(u_k)_{k \in \N}$ converges to $u$ in $\L^p$ and pointwisely almost everywhere on $[a,b]$ (up to a subsequence that we do not relabel, from the PCLDC theorem). Moreover, based on the construction by convolution product in~\cite[Corollary~4.23 p.109]{brezis}, the uniform bound~$\Vert u_k \Vert_{\L^\infty} \leq \Vert u \Vert_{\L^\infty}$ holds true for all $k \in \N$. Then we define $X_k := X(a)+\Im [u_k] $ which satisfies $X_k \in \CCC$ with~$\CDm[X_k]=u_k$ for all~$k \in \N$. In particular it holds that $\LL(x) \leq \LL(X_k)$ for all~$k \in \N$. Since~$p > \frac{1}{\a}$, one can easily deduce from~\cite[Property~4 p.242]{bourdin2013} that the sequence~$(X_k)_{k \in \N}$ uniformly converges on $[a,b]$ to~$X$. From the LDC theorem, one can get that~$\LL (X_k)$ tends to $\LL(X) $ when~$k \to \infty$. Finally we have proved that $\LL(x) \leq \LL(X)$ for all $X \in \CACmi$. Hence, if $x \in \CCC$ is a minimizer of~$\LL_{|\CCC}$, then $x$ is a minimizer of~$\LL$. As a consequence, the first- and second-order necessary optimality conditions derived in Theorem~\ref{thmmain1} are also valid.
\end{remark}

\begin{remark}\label{rem465}
In the case where $0 < \alpha < 1 \leq \beta$ (in particular for $\beta = 1$), note that the transversality condition~\eqref{eqtransv2} implies that $\partial_2 \varphi ( x(a) , x(b)) = 0_{\R^n}$ from the right counterpart of Proposition~\ref{prop3}. This remark allows us to conclude that a (quite) large class of purely fractional variational problems has no solution. Precisely, if one considers $0 < \alpha < 1 \leq \beta$ and some Mayer function $\varphi$ such that~$\partial_2 \varphi (x_1,x_2) \neq 0_{\R^n}$ for all $(x_1,x_2) \in \R^n \times \R^n$, then we can directly conclude from Theorem~\ref{thmmain1} that the corresponding Bolza functional~$\LL$ has no minimizer in $\CACmi$. This specific feature of the fractional setting is illustrated in Example~\ref{ex1} below.
\end{remark}

\begin{example}\label{ex1}
Let us consider the one-dimensional ($n=1$) Bolza functional given by
$$ \LL(x) := x (b) + \I^\beta_{a+} \left[ \dfrac{1}{2} \Big( x^2 + \CDm[x]^2 \Big) \right] ( b), $$
for all $x \in \CACmi$, where $0 < \alpha \leq 1$ and $\beta=1$, $a=0$ and $b=1$. In the classical case $\alpha=1$, it can be proved that $\LL$ admits a minimizer in ${}_\cc \AC^{1,\infty}_{a+}$ given by $x(t) = \frac{4e}{1-e^2} \cosh (t) $ for all~$t \in [a,b]$. Indeed one has to solve the Euler-Lagrange equation together with the transversality conditions in order to determinate the above candidate, and then prove that this candidate is optimal from the convexity of the Lagrange cost (following for example the strategy proposed in~\cite[p.258]{Brunt2004}). In contrast, we can directly conclude from Remark~\ref{rem465} that~$\LL$ has no minimizer in $\CACmi$ in the purely fractional case~$0 < \alpha < 1$.
\end{example}

\begin{remark}\label{rem464}
From~\cite[Property~4 p.242]{bourdin2013}, the content of Remark~\ref{rem465} and Example~\ref{ex1} can even be generalized to the case where $0 < \alpha < 1$ and $\beta > \alpha$.
\end{remark}

\begin{remark}\label{rem466}
In the spirit of Remark~\ref{rem465}, we refer to the paper~\cite{Ferreira2018} in which a detailed discussion is provided about the generalization of the calculus of variations to the fractional setting by preserving (or not) the existence of solutions.
\end{remark}

\subsection{Obstructions for the Legendre condition in the fractional setting with final constraints}\label{secobstruction}

As we mentioned in Remark~\ref{remcondfixee}, the first-order necessary optimality conditions in Theorem~\ref{thmmain1} can be similarly derived even by considering a fixed endpoint in Problem~\eqref{minproblem}. In this section our aim is to discuss in detail the obstructions that we encounter in deriving the Legendre condition in the fractional setting with final constraints. 

\medskip

Roughly speaking the standard proof of the Legendre condition in the classical case ($\alpha = 1$) is based on the existence of nontrivial variations $\eta \in {}_\cc \AC^{1,\infty}_{a+}$ such that:
\begin{enumerate}
\item[\rm{(i)}] $\eta$ and $ \dot{\eta}$ are compactly supported in a small interval $[\t,\t+h] \subset (a,b)$;
\item[\rm{(ii)}] $\dot{\eta}$ ``dominates" $\eta$ on $[\t,\t+h]$ in a sense to precise (see the discussion in~\cite[p.60]{Liberzon} for details).
\end{enumerate}
In particular, since the variation $\eta$ is compactly supported, it holds that $\eta(a)=\eta(b) = 0_{\R^n}$ which does not perturb the initial and final points (and thus, they can be considered as fixed). Several different families of variations have been considered in the literature (see, e.g., \cite[proof of Lemma p.103]{Fomin}, \cite[proof of Theorem~10.3.1 p.228]{Brunt2004} or~\cite[proof of Theorem~1.3 p.26]{Zelikin}). 

\medskip

In a first attempt to derive a fractional version of the Legendre condition (with fixed initial and final points), the authors of~\cite{Lazo2014} followed the same classical strategy as above. Unfortunately, one can easily check that the variation $\eta$ considered in \cite[Equality~(11)]{Lazo2014} does not satisfy all of the above properties. Precisely, while $\eta$ is indeed compactly supported, $\CDm [\eta]$ is clearly not (in contrary to what is claimed in \cite[Equality~(12)]{Lazo2014}). Surprisingly the same mistake has been disseminated in a series of papers (see~\cite{Almeida2017,Almeida2018,AlmeidaMorgado}). This discovery was the starting point of the present work. Actually, due to the very well-known \emph{memory skill} of the fractional derivative~$\CDm$, we conjectured that Property~(i) could not be satisfied by any nontrivial variation $\eta$ in the purely fractional case $0 < \alpha < 1$. This is exactly the content of the following novel result, which condemns for good the exact adaptation of the classical approach for the Legendre condition to the purely fractional case with final constraints.

\begin{proposition}\label{prop765}
Let $0 < \a <1$ and $x \in \CACm$. If there exist two real numbers $a \leq c < d \leq b$ such that:
\begin{enumerate}
\item[\rm{(i)}] $x(t)=0_{\R^n}$ for all $t \in [c,d]$;
\item[\rm{(ii)}] $\CDm [x](t) = 0_{\R^n}$ for almost every $t \in [c,d]$;
\end{enumerate}
then $x(t)=0_{\R^n}$ for all $t \in [a,d]$.
\end{proposition}

\begin{proof}
Since the proof of Proposition~\ref{prop765} is not trivial, we took the decision to move it to Appendix~\ref{app09}.
\end{proof}

\begin{remark}
We would like to emphasize that Proposition~\ref{prop765} is an intrinsic result of fractional calculus, in the sense that it is clearly not true for $\alpha = 1$.
\end{remark}

\begin{remark}\label{remfracdiff}
Together with the lemmas constituting its proof, Proposition~\ref{prop765} should be of independent interest for other researchers, in particular in the field of fractional differential equations. Indeed, it is well-known in classical differential equations that two different initial conditions yield two different solutions that cannot intersect each other. The preservation (or not) of this fundamental property to the purely fractional case was discussed in~\cite{cong,Diethelm}. In particular the authors of~\cite{cong} prove that this property is preserved in the one-dimensional setting, while it does not in the higher-dimensional case (a counter-example is provided). Note that Proposition~\ref{prop765} allows to contribute to this discussion. Precisely we can deduce that, even in the higher-dimensional case, two different initial conditions of a Caputo fractional differential equation of order $0 < \alpha < 1$ yield two different solutions that cannot coincide on an interval with a nonempty interior.
\end{remark}

The authors of~\cite{Lazo2014} considered a compactly supported variation $\eta$, believing that its fractional derivative~$\CDm [\eta]$ is also compactly supported (which is not). During our bibliographical search, we found the article~\cite{Guo2013} in which the Legendre condition is proved in the more general framework of fractional optimal control theory (with free endpoint). In contrast to~\cite{Lazo2014}, the author of~\cite{Guo2013} provides a ``reverse" strategy, in the sense that he considers first a compactly supported fractional derivative~$\CDm [\eta]$ and then obtains the corresponding variation $\eta$ by fixing $\eta (a)=0_{\R^n}$. With this approach, it is clear that $\eta$ is not compactly supported and, in particular, $\eta (b) \neq 0_{\R^n}$ (which has no repercussion since the endpoint is considered to be free in~\cite{Guo2013}). In the present work, our proof of the Legendre condition in Theorem~\ref{thmmain1} is based on the same approach than in~\cite{Guo2013} (see the definition of the variation~$\eta$ in~\eqref{eqguo}). 

\medskip

One may conclude that, if we do not consider final constraints in Problem~\eqref{minproblem}, then the strategy developed by~\cite{Guo2013} allows to prove the Legendre condition. However, showing that the Legendre condition in the fractional setting still holds considering final constraints remains an open challenge in the literature. Our aim in the next section is to fill this gap by considering the even more general framework of mixed initial/final constraints.

\begin{remark}\label{remguo1}
Though the variation considered in~\cite{Guo2013} allows to prove the Legendre condition in the fractional setting without final constraint, the proof developed in~\cite[Theorem~4.1 Step~B]{Guo2013} unfortunately contains mistakes in some estimations. Indeed, one may check that Equalities~(13) and~(14) of \cite[p.122]{Guo2013} are not correct. Nevertheless, these mistakes are not fatal and can be corrected (for example, one has to replace $O(\eps_1^2)$ by $O(\eps_1^{\alpha+1})$ in~(13), etc.). The proof of Lemma~\ref{lem3} in the present paper provides a corrected version of these estimations.
\end{remark}

\begin{remark}\label{remguo2}
In a second part of~\cite{Guo2013}, the author considers a fractional optimal control problem with final constraints. It is intriguing that the author considers the proof of this case to be similar to the one of the free endpoint case, omitting it in his work. It is by no means obvious for us how the author came to such desideratum. In the next section we provide a complete and detailed proof of the Legendre condition in the fractional setting with final constraints.
\end{remark}

\subsection{The case with general mixed initial/final constraints}\label{secwith}
In this section our aim is to deal with general mixed initial/final constraints in Problem~\eqref{minproblem}. Precisely we consider the functional constraint set
\begin{equation}\label{eqF}
\K := \{ x \in {}_\cc \AC^{\a,\infty}_{a+} \mid g(x(a),x(b)) \in \S \}, 
\end{equation}
where $j \geq 1$ is a positive integer, $g : \R^n \times \R^n \to \R^j$ is smooth of class $\C^1$ and $\S \subset \R^j $ is a nonempty closed convex subset of $\R^j$. We describe some typical situations of constraints in Remark~\ref{remarkconditionsterminales}. Let us recall the two following notions:
\begin{enumerate}
\item[\rm{(i)}] The map $g : \R^n \times \R^n \to \R^j$ is said to be \textit{regular} (or \textit{submersive}) at $(x_1,x_2) \in \R^n \times \R^n$ if its differential~$\mathcal{D} g(x_1,x_2)$ at this point is surjective;
\item[\rm{(ii)}] The \textit{normal cone} to $\S$ at a point $z \in\S$ is defined as the set 
$$ \mathrm{N}_\S[z] :=\{ z' \in\R^j \mid \forall z'' \in \S, \;  z' \cdot  (z'' -z)  \leq 0 \} .$$
For example, if $\S = \R^j$ is the entire space, then $\mathrm{N}_\S[z] = \{ 0_{\R^j} \}$ is reduced to the origin singleton for all $z \in \R^j$. On the other hand, if $\S = \{ \bar{z} \}$ is reduced to a singleton, then $\mathrm{N}_\S[\bar{z}] = \R^j$ is the entire space. 
\end{enumerate}
We are now in a position to state the main result of the present paper.

\begin{theorem}\label{thm2}
Let us assume that $\K$ is given by~\eqref{eqF}. If $x \in \K$ is a solution to Problem~\eqref{minproblem} and~$g$ is regular at $(x(a),x(b))$, then:
\begin{enumerate}
\item[\rm{(i)}] \textbf{Euler-Lagrange equation}: the function $\frac{(b-\cdot)^{\beta-1}}{\Gamma (\beta)} \partial_2 L(x,\CDm [x],\cdot) \in \AC^\a_{b+}$ with
\begin{equation*}\label{eqeuler-2}
\Dp \left[ \dfrac{(b-\cdot)^{\beta-1}}{\Gamma (\beta)} \partial_2 L(x,\CDm [x],\cdot)  \right](t) = - \dfrac{(b-t)^{\beta-1}}{\Gamma (\beta)} \partial_1 L(x(t),\CDm [x](t),t),
\end{equation*}
for almost every $t \in [a,b]$;
\item[\rm{(ii)}] \textbf{Transversality conditions}: the equalities
\begin{eqnarray}
\I^{1-\a}_{b-}  \left[ \dfrac{(b-\cdot)^{\beta-1}}{\Gamma (\beta)} \partial_2 L(x,\CDm [x],\cdot)  \right] (a) & = &  \partial_1 \varphi ( x(a) , x(b)) - \partial_1 g ( x(a),x(b) )^\top \times \psi ,\label{eqtransv1-2} \\
 - \I^{1-\a}_{b-}  \left[ \dfrac{(b-\cdot)^{\beta-1}}{\Gamma (\beta)} \partial_2 L(x,\CDm [x],\cdot)  \right] (b) & = & \partial_2 \varphi ( x(a) , x(b)) - \partial_2 g ( x(a),x(b) )^\top \times \psi,  \label{eqtransv2-2}
\end{eqnarray}
are both satisfied, where $\psi \in \R^j$ is such that $ - \psi \in \mathrm{N}_\S [ g(x(a),x(b)) ]$;
\item[\rm{(iii)}] \textbf{Legendre condition}: the matrix~$\frac{(b-t)^{\beta-1}}{\Gamma(\beta)} \partial^2_{22} L(x(t),\CDm [x](t),t) \in \R^{n \times n}$ is positive semi-definite for almost every $t \in [a,b]$.
\end{enumerate} 
\end{theorem}

\begin{proof}
As we mentioned before, the standard proof of the Legendre condition in the classical case~$\alpha = 1$ cannot be adapted to the fractional case $0< \alpha < 1$ when dealing with final constraints. We refer to Section~\ref{secobstruction} for details on the obstructions. In order to overcome this difficulty, our idea is to look for an alternative proof of the Legendre condition in the classical case $\alpha = \beta = 1$. It is well-known that optimal control theory can be seen as a generalization of the calculus of variations. Precisely, the Pontryagin maximum principle (in short, PMP), when applied to a calculus of variations problem, allows to derive the corresponding Euler-Lagrange equation, the transversality conditions and also the Legendre condition. As a consequence, our idea is to adapt the well-known proof of the PMP based on the Ekeland variational principle~\cite{Ekeland1974} to Problem~\eqref{minproblem}. We mention here that a PMP in the more general framework of fractional optimal control theory, containing additional functional constraints, has already been derived in~\cite{Bourdin2018} following the same strategy (with the convenient but unnecessary assumption $\beta \geq \alpha$). Since the detailed proof of Theorem~\ref{thm2} is quite long and technical, we took the decision to move it to Appendix~\ref{secproof}.
\end{proof}

A list of comments is in order.

\begin{remark}\label{remarkconditionsterminales}
Let us give the description of some typical situations of mixed initial/final constraints~$g(x(a),x(b)) \in \S$ in Problem~\eqref{minproblem}, and of the corresponding transversality conditions in Theorem~\ref{thm2}:
\begin{itemize}
\item[--] If the initial and final points are free in Problem~\eqref{minproblem}, one may consider $g$ as the identity function and $\S = \R^n \times \R^n$. In that case, the transversality conditions in Theorem~\ref{thm2} are given by
\begin{eqnarray*}
\I^{1-\a}_{b-}  \left[ \dfrac{(b-\cdot)^{\beta-1}}{\Gamma (\beta)} \partial_2 L(x,\CDm [x],\cdot)  \right] (a) & = &  \partial_1 \varphi ( x(a) , x(b)) , \\
 - \I^{1-\a}_{b-}  \left[ \dfrac{(b-\cdot)^{\beta-1}}{\Gamma (\beta)} \partial_2 L(x,\CDm [x],\cdot)  \right] (b) & = & \partial_2 \varphi ( x(a) , x(b)) . 
\end{eqnarray*}
\item[--] If the initial point is fixed to some $x_a \in \R^{n}$ and the final point is free in Problem~\eqref{minproblem}, one may consider $g$ as the identity function and $\S = \{ x_a \} \times \R^n$. In that case, the transversality condition~\eqref{eqtransv1-2} does not provide any additional information, while the transversality condition~\eqref{eqtransv2-2} gives 
$$ - \I^{1-\a}_{b-}  \left[ \dfrac{(b-\cdot)^{\beta-1}}{\Gamma (\beta)} \partial_2 L(x,\CDm [x],\cdot)  \right] (b) =  \partial_2 \varphi (x(a),x(b)) . $$
\item[--] If the initial and final points are fixed respectively to $x_a \in \R^{n}$ and $x_b \in \R^{n}$ in Problem~\eqref{minproblem}, one may consider $g$ as the identity function and $\S = \{ x_a \} \times \{ x_b \}$. In that case, the transversality conditions in Theorem~\ref{thm2} do not provide any additional information.
\item[--] If the initial point is fixed to some $x_a \in \R^n$ and the final point is subject to inequality constraints $G_i (x(b)) \leq 0$ for $i=1,\ldots,q$ for some $q \geq 1$, one may consider $g : \R^n \times \R^n \to \R^{n+q}$ defined by $g(x_1,x_2) : = (x_1,G(x_2))$ for all $x_1$, $x_2 \in \R^n$, where $G = (G_1,\ldots,G_q) : \R^n \rightarrow \R^q$ and $\S = \{ x_a \} \times (\R_-)^q$. If~$G$ is of class $\C^1$ and is regular at any point $x_2 \in G^{-1} ((\R_-)^q)$, then the transversality condition~\eqref{eqtransv1-2} does not provide any additional information, while the transversality condition~\eqref{eqtransv2-2} can be written as 
$$ - \I^{1-\a}_{b-}  \left[ \dfrac{(b-\cdot)^{\beta-1}}{\Gamma (\beta)} \partial_2 L(x,\CDm [x],\cdot)  \right] (b) = \partial_2 \varphi (x(a),x(b)) + \di \sum_{i=1}^q \lambda_i \nabla G_i (x(b)), $$
for some $\lambda_i \geq 0$ for all $i=1,\ldots,q$.
\item[--] If the periodic constraint $x(a)=x(b)$ is considered in Problem~\eqref{minproblem}, one may consider $g : \R^n \times \R^n \to \R^n$ defined by $g(x_1,x_2) := x_2 - x_1$ for all $x_1$, $x_2 \in \R^n$, and $\S = \{ 0_{\R^n} \}$. If moreover there is no Mayer cost in Problem~\eqref{minproblem} (that is, $\varphi = 0$), then the transversality conditions~\eqref{eqtransv1-2} and~\eqref{eqtransv2-2} provide the periodic equality
$$ \I^{1-\a}_{b-}  \left[ \dfrac{(b-\cdot)^{\beta-1}}{\Gamma (\beta)} \partial_2 L(x,\CDm [x],\cdot)  \right] (a) = \I^{1-\a}_{b-}  \left[ \dfrac{(b-\cdot)^{\beta-1}}{\Gamma (\beta)} \partial_2 L(x,\CDm [x],\cdot)  \right] (b). $$
\end{itemize}
In all above situations, we assert that $g$ is regular and then Theorem~\ref{thm2} can be applied.
\end{remark}

\begin{remark}
From Remark~\ref{remarkconditionsterminales}, one can easily see that Theorem~\ref{thm2} encompasses Theorem~\ref{thmmain1} and the related versions discussed in Remark~\ref{remcondfixee}. However, it is worth to note that Theorem~\ref{thmmain1} is derived from standard techniques of calculus of variations, while Theorem~\ref{thm2} (due to the presence of final constraints) requires a different strategy based on the adaptation of a more difficult and quite technical proof from optimal control theory based on the Ekeland variational principle (see Appendix~\ref{secproof}). 
\end{remark}

\begin{remark}
The proof of Theorem~\ref{thm2} only requires that $\varphi$ and $L$ are smooth of class $\C^1$ in order to derive the first-order necessary optimality conditions, and requires moreover that $L$ is twice-differentiable in order to derive the (second-order) Legendre condition.
\end{remark}

\begin{remark}
We emphasize here that the regularity assumption is not restrictive in Theorem~\ref{thm2}. Precisely, let $x \in \K$ be a solution to Problem~\eqref{minproblem}. If $g$ is not regular at $(x(a),x(b))$, one can replace $\K$ by $\overline{\K} := \{ X \in \CACmi \mid X(a)=x(a) \text{ and } X(b)=x(b) \}$. With this new set of constraints, it is clear that $x \in \overline{\K}$ and that $x$ is also a solution to the minimization problem associated to $\overline{\K}$ for which the regularity assumption is obviously satisfied (see Remark~\ref{remarkconditionsterminales}). With this strategy, Theorem~\ref{thm2} can be applied (but the transversality conditions do not provide any additional information, see Remark~\ref{remarkconditionsterminales}).
\end{remark}

\section{Conclusion}

This work provides a rigorous and thoughtful study of a constrained minimization problem of a general Bolza functional, depending on a Mayer cost, as well as on a fractional Lagrange cost of order~$\beta > 0$ involving a Caputo fractional derivative of order $0 < \alpha \leq 1$, under general mixed initial/final constraints. Motivated by our findings of some flaws within proofs of fractional versions of the second-order Legendre necessary optimality condition in previous works in the literature, the main contribution of the present article is to provide a (correct) proof based on the Ekeland variational principle of first- and second-order necessary optimality conditions, namely the Euler-Lagrange equation, the transversality conditions and, of course, the Legendre condition.

\medskip

Furthermore the present paper devotes an entire section (Section~\ref{secobstruction}) to discuss in detail the obstructions that are encountered to derive the Legendre condition in the fractional context with final constraints (which are not encountered in the classical context) and which are at the origin of the errors made in previous works in the literature. These obstructions led us to consider an alternative strategy based on the Ekeland variational principle, and moreover to derive new (nontrivial) results (see Proposition \ref{prop765} and Remark \ref{remfracdiff}) which may be of broad interest for the community of researchers in fractional calculus (not only in fractional calculus of variations), particularly within the fractional differential equations theory.

\medskip

Now a second-order necessary optimality condition in the field of fractional calculus of variations has been (correctly) derived in Theorem~\ref{thmmain1}, the most natural next perspective concerns the obtention of corresponding sufficient optimality conditions. Some papers in that direction have already been published (see, e.g., \cite{Herzallah2012,torres2011} and references therein). Numerous other challenges remain open in fractional calculus of variations. One of the most well-known problems, that still seems to resist the work of researchers, concerns the statement of fractional versions of \textit{constants of motion}. For example several attempts to derive fractional versions of the classical Noether's theorem have been provided in~\cite{Atan,Bourd2013,Frederico2010,Frederico2013,bale}. Unfortunately the results obtained are not fully satisfactory (for several reasons but essentially because they do not provide explicit and/or true constants of motion) and furthermore, like for the fractional extension of the Legendre condition considered in this paper, several papers contain fatal errors which invalidate the results therein (see, e.g., \cite{Ferreira2015} for counterexample and detailed discussion). A similar open challenge can be found in the field of fractional optimal control theory. Indeed it is well-known that the Hamiltonian function associated to a classical autonomous optimal control problem remains constant when evaluated over extremals but, as explained in~\cite[Section~5.1]{Bourdin2018}, the preservation of this property at the fractional level remains an open question. Finally, due to the nonlocal nature of fractional operators, several challenges have to be carried out in order to provide efficient numerical algorithms allowing to solve fractional variational problems (see~\cite[End of Section~5.1]{Bourdin2018} for more details). Note that several works in that direction have already been published (see, e.g., \cite{agrawal2,agrawal3,jelicic,ricardo2,ricardo}).

\section*{Acknowledgements}
Rui A. C. Ferreira was supported by the ``Funda\c{c}\~{a}o para a Ci\^encia e a Tecnologia (FCT)" through the program ``Stimulus of Scientific Employment, Individual Support-2017 Call" with reference CEECIND/00640/2017.

\appendix

\section{Proof of Proposition~\ref{prop765}}\label{app09}
In this section we provide a detailed proof of Proposition~\ref{prop765}. In order to accomplish it, recall that, for all~$\rho \in \R$, there exists a real sequence $(\rho_k)_{k \in \N}$ such that
\begin{equation}\label{eq1}
(1-\xi)^\rho = \sum_{k \in \N} \rho_k \xi^k,
\end{equation}
for all $\xi \in (-1,1)$. We also recall that the series convergence is uniform on all compact subsets included in $(-1,1)$. Moreover, if $\rho \in \R \backslash \N$, the terms $\rho_k$ are all different from zero.

\begin{lemma}\label{lem11}
Let $c > a$ be a real number and $u \in \mathrm{L}^1([a,c],\mathbb{R})$. If
\begin{equation*}
\int_a^c (s-a)^k u(s) \, ds = 0,
\end{equation*}
for all $k \in \N$, then $u=0$.
\end{lemma}

\begin{proof}
This result easily follows from the density of polynomial functions in $\C([a,c],\mathbb{R})$ and from~\cite[Corollary~4.24 p.110]{brezis}.
\end{proof}

\begin{lemma}\label{lem22}
Let $c > a$ be a real number and $u \in \mathrm{L}^1([a,c],\mathbb{R})$. Let us consider the function
$$ \fonction{\Psi}{(c,+\infty)}{\mathbb{R}}{t}{\Psi(t) := \di \int_a^c (t-s)^\mu u(s) \, ds ,} $$
where $\mu \in \R \backslash \N$. If the function $\Psi$ is polynomial over a subinterval $I \subset (c,+\infty)$ with a nonempty interior, then $u = 0$.
\end{lemma}

\begin{proof}
Without loss of generality, we can assume that $I = [c_1,c_2]$ is compact with $c < c_1 < c_2$. From the LDC theorem, one can easily see that $\Psi$ is of class~$\mathrm{C}^\infty$ with 
$$ \Psi^{(r)}(t) = \mu (\mu - 1) \ldots (\mu - r+1) \di \int_a^c (t-s)^{\mu-r} u(s) \, ds, $$
for all $t > c$ and all $r \geq 1$. In the sequel we fix some $r \geq 1$ larger than the degree of $\Psi$ (polynomial over $I$) plus one. Since~$\mu \in \R \backslash \N$, we get that
$$ \di \int_a^c (t-s)^{\mu-r} u(s) \, ds = 0, $$
and thus
$$ \di \int_a^c \left( 1 - \dfrac{s-a}{t-a} \right)^{\mu-r} u(s) \, ds = 0, $$
for all $t \in I$. From Equality~\eqref{eq1} (with $\rho := \mu-r \in \R \backslash \N$) and the uniform convergence of the power series (since~$0 \leq \frac{s-a}{t-a} \leq \frac{c-a}{c_1 - c } <1$ for all~$(t,s) \in I \times [a,c]$), we get that
$$ \di \sum_{k \in \N} \dfrac{\rho_k}{(t-a)^k} \int_a^c (s-a)^k u(s) \, ds = 0, $$
for all $t \in I$. Finally, using the change of variable $T=\frac{1}{t-a}$, we can write that $ \sum_{k \in \N} \lambda_k T^k = 0$ for all $T \in [\frac{1}{c_2 - a},\frac{1}{c_1-a}]$, where $ \lambda_k := \rho_k \int_a^c (s-a)^k u(s) \, ds $ for all~$k \in \N$. Since the zeros of a nonzero power series are isolated, we deduce that $\lambda_k = 0$ for all $k \in \N$. Since all~$\rho_k$ are different from zero, Lemma~\ref{lem11} concludes the proof.
\end{proof}

We are now in a position to prove Proposition~\ref{prop765}. Actually we can even prove the more general following statement.

\begin{proposition}\label{prop765bis}
Let $0 < \a <1$ and $x \in \CACm$. If there exist two real numbers $a \leq c < d \leq b$ such that:
\begin{enumerate}
\item[\rm{(i)}] $x$ is polynomial over $[c,d]$;
\item[\rm{(ii)}] $\CDm [x](t) = 0_{\R^n}$ for almost every $t \in [c,d]$;
\end{enumerate}
then $x$ is constant over $[a,d]$.
\end{proposition}

\begin{proof}
Without loss of generality, we assume in this proof that $n=1$. From Proposition~\ref{propcompC}, it holds that $ x(t) = x(a) + \Im [ \CDm [x] ](t)$ for all $t \in [a,b]$, and thus
$$ x(t) = x(a) + \di \int_a^c \dfrac{(t-s)^{\a-1}}{\Gamma(\a)}  \CDm [x](s) \, ds, $$
for all $t \in [c,d]$. Let us denote by $u \in \L^1([a,c],\R)$ the restriction of $\CDm [x]$ over $[a,c]$. From the hypothesis, one can easily deduce that the function
$$ \fonction{\Psi}{(c,+\infty)}{\mathbb{R}}{t}{\Psi(t) := \di \int_{a}^{c} (t-s)^{\alpha-1} u(s) \, ds ,} $$
is polynomial over $(c,d]$. From Lemma~\ref{lem22}, we deduce that $u=0$ and thus $\CDm [x](t) = 0$ for almost every $t \in [a,d]$. We deduce that $x(t)=x(a)$ for all $t \in [a,d]$ which completes the proof.
\end{proof}

\section{Proof of Theorem~\ref{thm2}}\label{secproof}
Our strategy of proof is based on the Ekeland variational principle~\cite{Ekeland1974}. Let us enunciate hereafter a simplified version (but sufficient for our purposes).

\begin{proposition}[Ekeland variational principle]\label{propekeland}
Let $(\E,\d_\E)$ be a complete metric space and $\J : \E \rightarrow \R^+$ be a continuous nonnegative map. Let $\eps>0$ and $\lambda \in \E$ such that $\J(\lambda) = \eps$. Then there exists $\lambda_\eps \in \E$ such that $\d_\E ( \lambda_\eps , \lambda ) \leq \sqrt{\eps}$ and $-\sqrt{\eps} \; \d_\E ( \lambda' , \lambda_\eps ) \leq \J(\lambda')-\J(\lambda_\eps)$ for all $\lambda' \in \E$.
\end{proposition}

In Section~\ref{secconvex} we give some recalls about convex analysis. In Section~\ref{secsens} we investigate the sensitivity analysis of the Bolza functional~$\LL$. Finally the proof of Theorem~\ref{thm2} is detailed in Section~\ref{secproofproof} by applying the Ekeland variational principle on a penalized functional.

\subsection{Basics of convex analysis}\label{secconvex}
Let $\d_\S : \R^j \to \R_{+}$ denote the standard distance function to the nonempty closed convex subset $\S \subset \R^j$ defined by $\d_\S(z):= \inf_{z' \in \S} \Vert z-z' \Vert_{\R^j}$ for all $z \in \R^j$. We recall that, for all $z \in \R^j$, there exists a unique element $\PP_\S(z) \in \S$ (called the \textit{projection} of~$z$ onto~$\S$) such that $\d_\S(z)=\Vert z-\PP_\S(z) \Vert_{\R^j}$. It can easily be shown that the map $\PP_\S:\R^j \rightarrow \S$ is $1$-Lipschitz continuous. Moreover it holds that~$( z- \PP_\S(z)) \cdot (z'-\PP_\S(z)) \leq0$ for all $z'\in\S$, that is, $z-\PP_\S(z)\in \mathrm{N}_\S[\PP_\S(z)]$ for all $z \in \R^j$. Let us recall the two following required lemmas, whose proofs are detailed for the reader's convenience.

\begin{lemma}\label{lemmeprojection}
Let $(z_k)_{k\in\N}$ be a sequence in $\R^j$ converging to some point $z\in\S$ and let $(\zeta_k)_{k\in\N}$ be a positive real sequence. If $\zeta_k(z_k-\PP_\S(z_k))$ converges to some $\overline{z} \in \R^j$, then $\overline{z} \in \mathrm{N}_\S[z]$.
\end{lemma}

\begin{proof}
Since~$z_k - \PP_\S(z_k) \in \mathrm{N}_\S[\PP_\S(z_k)]$ and~$\zeta_k > 0$ for all~$k \in \N$, we obtain that~$ \zeta_k (z_k - \PP_\S(z_k)) \cdot (z' - \PP_\S(z_k)) \leq 0$ for all~$z' \in \S$ and all~$k \in \N$. Passing to the limit~$k \to \infty$, and since~$z \in \S$, we obtain that~$ \overline{z} \cdot (z' - z) \leq 0$ for all~$z' \in \S$ which exactly means that~$\overline{z} \in \mathrm{N}_\S[z]$.
\end{proof}

\begin{lemma}\label{lemmenormalcone}
The map 
$$ \fonction{\d^2_\S}{\R^j}{\R_+}{z}{\d^2_\S(z) := \d_\S(z)^2 ,}$$ 
is Fr\'echet-differentiable on $\R^j$, and its differential $\mathcal{D}\d^2_\S(z)$ at every $z \in \R^j$ can be expressed as 
$$ \mathcal{D}\d^2_\S(z)(z') = 2  ( z-\PP_\S(z) ) \cdot z'  , $$ 
for all $z' \in\R^j$.
\end{lemma}

\begin{proof}
Let~$z \in \R^j$ and let us prove that $\d^2_\S$ is Fr\'echet-differentiable at~$z$ with $ \mathcal{D}\d^2_\S(z)(z') = 2  ( z-\PP_\S(z) ) \cdot z'$. One has
$$
\d^2_\S(z+z') - \d^2_\S(z) 
 \leq \Vert z+z' - \PP_\S(z) \Vert^2 - \Vert z - \PP_\S (z) \Vert^2 =  2  ( z-\PP_\S(z) ) \cdot z' + \Vert z' \Vert^2, 
$$
and, from Cauchy-Schwarz inequality and~$1$-Lipschitz continuity of~$\PP_\S$, one gets
\begin{multline*}
\d^2_\S(z) - \d^2_\S(z+z') 
 \leq \Vert z - \PP_\S(z+z') \Vert^2 - \Vert z+z' - \PP_\S (z+z') \Vert^2 \\ =  -2  ( z-\PP_\S(z+z') ) \cdot z' - \Vert z' \Vert^2 = -2  ( z-\PP_\S(z) ) \cdot z' + 2 ( \PP_\S(z+z')-\PP_\S(z)) \cdot z' - \Vert z' \Vert^2 \\
  \leq -2  ( z-\PP_\S(z) ) \cdot z' + \Vert z' \Vert^2 ,
\end{multline*}
for all~$z' \in \R^j$. Using both inequalities, the proof is complete.
\end{proof}

\subsection{Sensitivity analysis of the Bolza functional}\label{secsens}
In the proof of Theorem~\ref{thm2} (see Section~\ref{secproofproof} below), we denote:
\begin{itemize}
\item[--] by $r_\alpha$ some real number satisfying $r_\alpha > \frac{1}{\alpha}$ and by $r'_\alpha := \frac{r_\alpha}{r_\alpha - 1}$ the classical conjugate of $r_\alpha$ satisfying $\frac{1}{r_\alpha} + \frac{1}{r'_\alpha} = 1$;
\item[--] and, for all $(u,y) \in \L^\infty \times \R^n$, by $x(\cdot,u,y) \in {}_\cc \AC^{\a,\infty}_{a+}$ the function defined by
$$ x(t,u,y) := y + \Im [u](t), $$
for all $t \in [a,b]$;
\item[--] and, for all $(u,y) \in \L^\infty \times \R^n$, by $\PPP(u,y)$ the set of Lebesgue points $\t \in (a,b)$ of both the functions~$u$ and $L(x(\cdot,u,y),u,\cdot)$.
\end{itemize}

\begin{remark}
Note that $r'_\alpha (\alpha-1) +1 > 0$.
\end{remark}

\begin{remark}
Let $x \in \L^1$. From Proposition~\ref{propcompC} and Remark~\ref{rempropcompC2}, note that $x \in {}_\cc \AC^{\a,\infty}_{a+}$ if and only if there exists $(u,y) \in \L^\infty \times \R^n$ such that $x=x(\cdot,u,y)$. In that case, the couple $(u,y)$ is unique and is given by $u = \CDm[x]$ and~$y=x(a)$.
\end{remark}

\begin{remark}
For all $(u,y) \in \L^\infty \times \R^n$, note that the set $\PPP(u,y)$ is of full measure in $[a,b]$.
\end{remark}

We introduce the set
$$ \L^\infty_R := \L^\infty \Big( [a,b] , \BB_{\R^n}(0_{\R^n},R ) \Big), $$
for all $R \geq 0$, where $\BB_{\R^n}(0_{\R^n},R )$ denotes the standard closed ball of $\R^n$ centered at the origin $0_{\R^n}$ with radius $R \geq 0$. We endow the set $\L^\infty_R \times \R^n$ with the distance
$$ \d_{ \L^\infty_R \times \R^n } \Big( (u_2,y_2) , (u_1 , y_1) \Big) := \Vert u_2 - u_1 \Vert_{\L^1} + \Vert y_2 - y_1 \Vert_{\R^n}, $$
for all $(u_1,y_1)$, $(u_2,y_2) \in \L^\infty_R \times \R^n$.

\begin{lemma}\label{lemcomplete}
Let $R \geq 0$. The following assertions are true:
\begin{enumerate}
\item[\rm{(i)}] The metric space $( \L^\infty_R \times \R^n , \d_{ \L^\infty_R \times \R^n } )$ is complete;
\item[\rm{(ii)}] The map $(u,y) \in \L^\infty_R \times \R^n \longmapsto x(\cdot,u,y) \in \C $ is continuous;
\item[\rm{(iii)}] The map
$$ \fonction{\Phi_R}{\L^\infty_R \times \R^n}{\R}{(u,y)}{\Phi_R (u,y) := \LL (x(\cdot,u,y)),} $$
is continuous.
\end{enumerate}
\end{lemma}

\begin{proof}
The first item can be easily derived from the PCLDC theorem. Secondly, it can be proved from the classical H\"older inequality that
$$ \Vert x(t,u_2,y_2) -  x(t,u_1,y_1) \Vert_{\R^n} \leq \Vert y_2 - y_1 \Vert_{\R^n} +\dfrac{1}{\Gamma(\alpha)} \left( \dfrac{ R (b-a)^{r'_\alpha (\alpha-1) +1}}{r'_\alpha (\alpha-1) +1 } \right)^{1/r'_\alpha} \Vert u_2 - u_1 \Vert^{1/r_\alpha}_{\L^1} ,$$
for all $t \in [a,b]$ and all $(u_1,y_1)$, $(u_2,y_2) \in \L^\infty_R \times \R^n$, which concludes the proof of the second item. The third item can be derived by contradiction and by using the LDC and PCLDC theorems.
\end{proof}

The rest of this section is devoted to the sensitivity analysis of the Bolza functional $\Phi_R$ under perturbations of the couple $(u,y)$ (see Propositions~\ref{propneedle} and~\ref{propcont}). Before coming to these points, we first introduce the following notion of needle-perturbation of~$u$.

\begin{definition}[Needle-perturbation of $u$]
Let $R \geq 0$ and let $(u,y) \in \L^\infty_R \times \R^n$. A needle-perturbation of $u$ associated to $(\t,v) \in \PPP(u,y) \times \BB_{\R^n}(0_{\R^n},R)$ and $0 < h \leq b-\t$ is the function~$u^{(\t,v)}(\cdot,h) \in  \L^\infty_R$ defined by
$$ u^{(\t,v)}(t,h) := \left\lbrace \begin{array}{lcl}
v & \text{if} & t \in [\t,\t+h) , \\
u(t) & \text{if} & t \notin [\t,\t+h) , 
\end{array}
\right. $$
for almost every $t \in [a,b]$.
\end{definition}

\begin{lemma}\label{lemneedletraj}
Let $R \geq 0$ and let $(u,y) \in \L^\infty_R \times \R^n$. It holds that
\begin{equation}\label{ineq891}
\left\Vert x \Big( t ,u^{(\t,v)}(\cdot,h),y \Big)- x(t,u,y) \right\Vert_{\R^n} \leq \dfrac{2R}{\Gamma(\a+1)} h^\alpha,
\end{equation}
for all $t \in [a,b]$ and
\begin{multline}\label{ineq890}
\left\Vert \dfrac{ x \Big( t , u^{(\t,v)}(\cdot,h),y \Big)- x(t,u,y)}{h} - \dfrac{(t-\t)^{\alpha-1}}{\Gamma(\alpha)}  (v-u(\t)) \right\Vert_{\R^n} \\
\leq \dfrac{( t - \t )^{\a-1}}{\Gamma(\a)}  \left\Vert \dfrac{1}{h} \int_\t^{\t+h} u(s) \; ds - u(\t) \right\Vert_{\R^n} 
+\dfrac{2R}{\Gamma(\a)} \Bigg( \Big( t - (\t+h) \Big)^{\a-1} - (t-\t)^{\a-1} \Bigg) ,
\end{multline}
for all $t \in (\t+h,b]$, all $(\t,v) \in \PPP(u,y) \times \BB_{\R^n}(0_{\R^n},R)$ and all $0 < h \leq b-\t$.
\end{lemma}

\begin{proof}
Let~$(\t,v) \in \PPP(u,y) \times \BB_{\R^n}(0_{\R^n},R)$ and $0 < h \leq b-\t$ being fixed for the whole proof. It holds that
\begin{multline*}
 x \Big( t ,u^{(\t,v)}(\cdot,h),y \Big)- x(t,u,y)  =  \I^\alpha_{a+} [ u^{(\t,v)}(\cdot,h) -u ](t) \\
 = \int_a^t \dfrac{(t-s)^{\alpha - 1}}{\Gamma (\alpha)} \Big( u^{(\t,v)}(s,h) -u(s) \Big) \; ds 
 = \left\lbrace \begin{array}{lcl}
 0_{\R^n} & \text{if} & a \leq t < \tau , \\[10pt]
\di \int_{\tau}^t \dfrac{(t-s)^{\alpha - 1}}{\Gamma (\alpha)} ( v -u(s) ) \; ds & \text{if} &  \tau \leq t \leq \tau + h , \\[15pt]
\di \int_{\tau}^{\tau + h} \dfrac{(t-s)^{\alpha - 1}}{\Gamma (\alpha)} ( v -u(s) ) \; ds & \text{if} &   t > \tau + h ,
 \end{array}
  \right.
\end{multline*}
for all~$t \in [a,b]$. Since~$u \in \L^\infty_R$ and~$v \in \BB_{\R^n}(0_{\R^n},R)$, we obtain that
\begin{multline*}
\left\Vert x \Big( t ,u^{(\t,v)}(\cdot,h),y \Big)- x(t,u,y) \right\Vert_{\R^n} 
 \leq  \left\lbrace \begin{array}{lcl}
 0 & \text{if} & a \leq t < \tau , \\[10pt]
2 R \di \int_{\tau}^t \dfrac{(t-s)^{\alpha - 1}}{\Gamma (\alpha)} \; ds & \text{if} &  \tau \leq t \leq \tau + h , \\[15pt]
2R \di \int_{\tau}^{\tau + h} \dfrac{(t-s)^{\alpha - 1}}{\Gamma (\alpha)}  \; ds & \text{if} &   t > \tau + h ,
 \end{array}
  \right.
  \\[10pt]
  =  \left\lbrace \begin{array}{lcl}
 0 & \text{if} & a \leq t < \tau , \\[10pt]
2 R \dfrac{(t-\tau)^\alpha}{\Gamma(\alpha+1)} & \text{if} &  \tau \leq t \leq \tau + h , \\[15pt]
2R \dfrac{(t-\tau)^\alpha - (t-(\tau+h))^\alpha}{\Gamma(\alpha+1)} & \text{if} &   t > \tau + h ,
 \end{array}
  \right.
\end{multline*}
for all~$t \in [a,b]$. To prove Inequality~\eqref{ineq891}, one has just to see that, in all of the three above cases, the right-hand side term is less than~$\frac{2R}{\Gamma(\alpha+1)}h^\alpha$. For the last case, one has just to invoke the basic inequality~$\chi_2^\alpha - \chi_1^\alpha \leq (\chi_2 - \chi_1 )^\alpha$ which is satisfied for all~$0 \leq \chi_1 \leq \chi_2$ (see the proof of Lemma~\ref{lem30} for some details). Now let us prove Inequality~\eqref{ineq890}. Using similar arguments, one has
\begin{multline*}
\dfrac{ x \Big( t , u^{(\t,v)}(\cdot,h),y \Big)- x(t,u,y)}{h} - \dfrac{(t-\t)^{\alpha-1}}{\Gamma(\alpha)}  (v-u(\t)) \\
= \dfrac{1}{h} \int_\t^{\t+h} \dfrac{(t-s)^{\alpha-1}}{\Gamma(\alpha)} (v-u(s)) \; ds - \dfrac{(t-\t)^{\alpha-1}}{\Gamma (\alpha)} (v-u(\t)) \\[5pt]
=  \dfrac{1}{h} \int_\t^{\t+h} \dfrac{(t-s)^{\alpha-1} -(t-\t)^{\alpha-1} }{\Gamma(\alpha)} (v-u(s)) \; ds +  \dfrac{(t-\t)^{\alpha-1}}{\Gamma(\alpha)} \left( u(\t) - \dfrac{1}{h} \int_\t^{\t+h} u(s) \; ds \right) ,
\end{multline*}
for all~$t \in (\t+h,b]$. Since~$u \in \L^\infty_R$ and~$v \in \BB_{\R^n}(0_{\R^n},R)$, we obtain that
\begin{multline*}
\left\Vert \dfrac{ x \Big( t , u^{(\t,v)}(\cdot,h),y \Big)- x(t,u,y)}{h} - \dfrac{(t-\t)^{\alpha-1}}{\Gamma(\alpha)}  (v-u(\t)) \right\Vert_{\R^n} \\
\leq \dfrac{2R}{\Gamma(\a)} \dfrac{1}{h} \int_\t^{\t+h} (t-s)^{\alpha-1} -(t-\t)^{\alpha-1} \; ds + \dfrac{( t - \t )^{\a-1}}{\Gamma(\a)}  \left\Vert \dfrac{1}{h} \int_\t^{\t+h} u(s) \; ds - u(\t) \right\Vert_{\R^n}  ,
\end{multline*}
for all~$t \in (\t+h,b]$. Noting that~$(t-s)^{\alpha-1} \leq (t-(\t+h))^{\alpha-1}$ for all~$s \in [\t,\t+h]$ and all~$t \in (\t+h,b]$, the proof of Inequality~\eqref{ineq890} is complete.
\end{proof}

\begin{proposition}[Sensitivity analysis under needle-perturbation of $u$]\label{propneedle}
Let $R \geq 0$ and let $(u,y) \in \L^\infty_R \times \R^n$. It holds that
\begin{multline*}
 \lim\limits_{h \to 0^+} \dfrac{\Phi_R \Big( u^{(\t,v)}(\cdot,h),y \Big)-\Phi_R(u,y)}{h} =  \dfrac{(b-\t)^{\beta-1}}{\Gamma(\beta)}  \Big( L(x(\t),v,\t)-L(x(\t),u(\t),\t) \Big) \\
  + \left( \dfrac{ (b-\t)^{\alpha-1}}{\Gamma(\alpha)} \partial_2 \varphi (x(a),x(b))  + \I^\alpha_{b-} \left[ \dfrac{(b-\cdot)^{\beta-1}}{\Gamma(\beta)}  \partial_1 L (x,u,\cdot) \right](\t) \right) \cdot (v-u(\t)) ,
\end{multline*}
for all $(\t,v) \in \PPP(u,y) \times \BB_{\R^n}(0_{\R^n},R)$, where $x = x(\cdot,u,y) \in {}_\cc \AC^{\a,\infty}_{a+}$. 
\end{proposition}

\begin{proof}
Let $(\t,v) \in \PPP(u,y) \times \BB_{\R^n}(0_{\R^n},R)$. For simplicity of notations in this proof, we denote by
$$ u_h := u^{(\t,v)}(\cdot,h) \quad \text{and} \quad x_h := x \Big( \cdot , u^{(\t,v)}(\cdot,h),y \Big),  $$
for all $0 < h \leq b-\t$. From Inequality~\eqref{ineq890} and since $\t$ is a Lebesgue point of $u$, it is clear that
$$ \lim\limits_{h \to 0^+} \dfrac{\varphi (x_h(a),x_h(b))-\varphi(x(a),x(b))}{h} = \dfrac{ (b-\t)^{\alpha-1}}{\Gamma(\alpha)} \partial_2 \varphi (x(a),x(b)) \cdot (v-u(\t)).  $$
Moreover the term
$$ \dfrac{\I^{\beta}_{a+} [ L(x_h,u_h,\cdot)](b) - \I^{\beta}_{a+} [ L(x,u,\cdot) ](b) }{h}, $$
can be decomposed as
\begin{multline*}
\dfrac{1}{h} \int_\t^{\t+h} \dfrac{(b-s)^{\beta-1}}{\Gamma(\beta)} \Big( L(x_h(s),v,s) - L(x(s),v,s) \Big) \; ds \\
+ \dfrac{1}{h} \int_\t^{\t+h} \dfrac{(b-s)^{\beta-1}}{\Gamma(\beta)} \Big( L(x(s),v,s) - L(x(s),u(s),s) \Big) \; ds \\
+ \dfrac{1}{h} \int_{\t+h}^b \dfrac{(b-s)^{\beta-1}}{\Gamma(\beta)} \Big( L(x_h(s),u(s),s) - L(x(s),u(s),s) \Big) \; ds,
\end{multline*}
for all $0 < h \leq b-\t$. From the uniform convergence obtained in Inequality~\eqref{ineq891}, the first term tends to zero when $h \to 0^+$. Since $\t$ is a Lebesgue point of the function $L(x,u,\cdot)$, it is clear that the second term tends to
$$ \dfrac{(b-\t)^{\beta-1}}{\Gamma(\beta)}  \Big( L(x(\t),v,\t)-L(x(\t),u(\t),\t) \Big), $$
when $h \to 0^+$. From a Taylor expansion with integral rest, the last term can be decomposed as
\begin{multline*}
 \int_{\t+h}^b \int_0^1 \dfrac{(b-s)^{\beta-1}}{\Gamma(\beta)}  \partial_1 L (\star_\theta) \cdot \left( \dfrac{x_h(s)-x(s)}{h} - \dfrac{(s-\t)^{\alpha-1}}{\Gamma(\alpha)} (v-u(\t)) \right) \; d\theta ds \\
+  \int_{\t+h}^b \int_0^1 \dfrac{(b-s)^{\beta-1}}{\Gamma(\beta)} \dfrac{(s-\t)^{\alpha-1}}{\Gamma(\alpha)} ( \partial_1 L (\star_\theta) - \partial_1 L (\star_0) ) \cdot (v-u(\t)) \; d\theta ds \\
 +  \int_{\t+h}^b  \dfrac{(b-s)^{\beta-1}}{\Gamma(\beta)} \dfrac{(s-\t)^{\alpha-1}}{\Gamma(\alpha)} \partial_1 L (\star_0)  \cdot (v-u(\t)) \; ds
\end{multline*}
where
$$ \star_\theta := ( x(s)+\theta (x_h(s)-x(s)) , u(s) , s), $$
for all $\theta \in [0,1]$ and all $0 < h \leq b-\t$. The last above term clearly tends to
$$ \I^\alpha_{b-} \left[ \dfrac{(b-\cdot)^{\beta-1}}{\Gamma(\beta)}  \partial_1 L (x,u,\cdot) \right](\t) \cdot (v-u(\t)), $$
when $h \to 0^+$. From the LDC theorem, the second above term clearly tends to zero when $h \to 0^+$. Finally one can prove from Inequality~\eqref{ineq890} that the norm of the first above term can be bounded by
\begin{multline*}
\dfrac{\Big( b-(\t+h) \Big)^{\a+\beta-1}}{\Gamma (\a+\beta)}  \left\Vert \dfrac{1}{h} \int_\t^{\t+h} u(s) \; ds - u(\t) \right\Vert_{\R^n} \\
+ \dfrac{2R}{\Gamma (\a+\beta)} \left( \Big( b-(\t+h) \Big)^{\a+\beta-1} - ( b-\t)^{\a+\beta-1} \right) 
+ 2R \int_\t^{\t+h} \dfrac{(b-s)^{\beta-1}}{\Gamma(\beta)} \dfrac{(s-\t)^{\alpha-1}}{\Gamma(\alpha)} \; ds, 
\end{multline*}
which tends to zero when $h \to 0^+$ (in particular since $\t$ is a Lebesgue point of the function $u$). The proof is complete.
\end{proof}

\begin{proposition}[Sensitivity analysis under perturbation of $y$]\label{propcont}
Let $R \geq 0$ and let $(u,y) \in \L^\infty_R \times \R^n$. It holds that
$$
\lim\limits_{h \to 0^+} \dfrac{\Phi_R(u,y+hy')-\Phi_R(u,y)}{h} 
= \left( \partial_1 \varphi (x(a),x(b)) + \partial_2 \varphi (x(a),x(b)) + \I^\beta_{a+} \Big[ \partial_1 L (x, u, \cdot) \Big](b) \right) \cdot y',
$$
for all $y' \in \R^n$, where $x = x(\cdot,u,y) \in {}_\cc \AC^{\a,\infty}_{a+}$.
\end{proposition}

\begin{proof}
We apply Proposition~\ref{propfirstdifferential} with the constant variation $\eta = y' \in \CACmi$.
\end{proof}

\subsection{Proof of Theorem~\ref{thm2} by applying the Ekeland variational principle}\label{secproofproof}
Let $x \in \K$ be a solution to Problem~\eqref{minproblem}. Using the notations introduced in Section~\ref{secsens}, it holds that $x = x(\cdot,u,y)$ where $u = \CDm[x]$ and $y=x(a)$. In particular note that $g(y,x(b,u,y))=g(x(a),x(b)) \in \S$. Let $R := \Vert u \Vert_{\L^\infty} + 1 \geq 0$ and let us consider a positive sequence $(\eps_k)_{k \in \N}$ which tends to zero when $k \to \infty$. We introduce the penalized functional
$$ \fonction{\J_k}{\L^\infty_R \times \R^n}{\R_+}{(u',y')}{ \sqrt{ \Bigg( \Big( \Phi_R (u',y') - \Phi_R (u,y) + \eps_k \Big)^+ \Bigg)^2 + \d^2_\S \Big( g(y',x(b,u',y')) \Big) } ,} $$
for all $k \in \N$. From Lemma~\ref{lemcomplete} and the continuities of $g$ and $\d^2_\S$, it is clear that $J_k$ is a continuous nonnegative map defined on a complete metric space for all $k \in \N$ . Since $\J_k (u,y) = \eps_k$ for all $k \in \N$, we deduce from the Ekeland variational principle (see Proposition~\ref{propekeland}) that there exists a sequence $(u_k,y_k)_{k \in \N} \subset \L^\infty_R \times \R^n$ such that 
$$ \d_{\L^\infty_R \times \R^n} \Big( (u_k,y_k),(u,y) \Big) \leq \sqrt{\eps_k}, $$
and
\begin{equation}\label{ineq}
- \sqrt{\eps_k} \; \d_{\L^\infty_R \times \R^n} \Big( (u',y') , (u_k,y_k) \Big) \leq \J_k (u',y') - \J_k(u,y),
\end{equation}
for all $(u',y') \in \L^\infty_R \times \R^n$ and all $k \in \N$. In the sequel, we denote by $x_k := x(\cdot,u_k,y_k)$ for all $k \in \N$. Note that the sequence $(u_k,y_k)_{k \in \N}$ converges to $(u,y)$ in $\L^\infty_R \times \R^n$, and thus the sequence $(x_k)_{k \in \N}$ converges to $x$ in $\C$ (see Lemma~\ref{lemcomplete}).

\medskip

From the optimality of $x$, one can easily see that $\J_k (u',y') > 0$ for all $(u',y') \in \L^\infty_R \times \R^n$ and all~$k \in \N$. As a consequence, we can correctly define
$$ \psi^0_k := - \dfrac{1}{\J_k (u_k,y_k)} \Big( \Phi_R (u_k,y_k) - \Phi_R (u,y) + \eps_k \Big)^+ \leq 0, $$
and
$$ \psi_k := - \dfrac{1}{\J_k (u_k,y_k)} \Bigg( g(x_k(a),x_k(b)) - \PP_\S \Big( g(x_k(a),x_k(b)) \Big) \Bigg) \in \R^j, $$
which satisfy $\vert \psi^0_k \vert^2 + \Vert \psi_k \Vert_{\R^j}^2 = 1$ for all $k \in \N$. From a standard compactness argument and from the PCLDC theorem, we can extract subsequences (that we do not relabel) such that $(\psi^0_k)_{k \in \N}$ converges to some $\psi^0 \leq 0$, $(\psi_k)_{k \in \N}$ converges to some $\psi \in \R^j $ satisfying $-\psi \in \mathrm{N}_\S [ g(x(a),x(b)) ]$ (see Lemma~\ref{lemmeprojection}) and $(u_k)_{k \in \N}$ converges to $u$ pointwisely almost everywhere on~$[a,b]$. Moreover, note that $\vert \psi^0 \vert^2 + \Vert \psi \Vert_{\R^j}^2 = 1$ and thus the couple $(\psi^0,\psi)$ is not trivial.

\paragraph{Perturbation of $y_k$.} Let $y' \in \R^n$ and let us fix some $k \in \N$. From Inequality~\eqref{ineq}, it holds that
\begin{equation*}
- \sqrt{\eps_k} \; \Vert y' \Vert_{\R^n} \leq \dfrac{1}{\J_k (u_k,y_k+hy') + \J_k (u_k,y_k)} \times \dfrac{\J^2_k ( u_k,y_k+hy' ) - \J^2_k (u_k,y_k)}{h}, 
\end{equation*}
for all $h > 0$. Letting $h \to 0^+$, we exactly get from Proposition~\ref{propcont} that
\begin{multline*}
 \Bigg(  \psi^0_k \left( \partial_1 \varphi (x_k(a),x_k(b)) + \partial_2 \varphi (x_k(a),x_k(b)) + \I^\beta_{a+} \Big[ \partial_1 L (x_k, u_k, \cdot) \Big](b) \right) \\
  + \Big( \partial_1 g ( x_k(a) , x_k(b) )^\top  + \partial_2 g ( x_k(a) , x_k(b) )^\top \Big) \times \psi_k \Bigg) \cdot y' \leq \sqrt{\eps_k} \Vert y' \Vert_{\R^n}.
\end{multline*}
Finally, letting $k \to \infty$, we obtain (using in particular the LDC theorem) that
\begin{multline*}
 \Bigg(  \psi^0 \Big( \partial_1 \varphi (x(a),x(b)) +  \partial_2 \varphi (x(a),x(b)) +  \I^\beta_{a+} \Big[ \partial_1 L (x, u, \cdot) \Big](b) \Big) \\
  + \Big( \partial_1 g ( x(a) , x(b) )^\top  + \partial_2 g ( x(a) , x(b) )^\top \Big) \times \psi \Bigg) \cdot y' \leq 0.
\end{multline*}
Since the above inequality is satisfied for all $y' \in \R^n$ and since we can write 
$$ \I^\beta_{a+} \Big[ \partial_1 L (x, u, \cdot) \Big](b) = \I^1_{b-} \left[ \dfrac{(b-\cdot)^{\beta-1}}{\Gamma(\beta)}  \partial_1 L (x,u,\cdot) \right](a), $$
we deduce the crucial equality given by
\begin{multline}\label{eqcrucial2}
\Big( \psi^0 \partial_2 \varphi (x(a),x(b)) +  \partial_2 g ( x(a) , x(b) )^\top \times \psi \Big) + \psi^0  \I^1_{b-} \left[ \dfrac{(b-\cdot)^{\beta-1}}{\Gamma(\beta)}  \partial_1 L (x,u,\cdot) \right](a) \\[8pt]
    = - \psi^0 \partial_1 \varphi (x(a),x(b))- \partial_1 g ( x(a) , x(b) )^\top \times \psi .
\end{multline}

\paragraph{Needle-perturbation of $u_k$.} Let $(\t,v) \in \PPP \times  \BB_{\R^n}(0_{\R^n},R)$ where
$$ \PPP := \Big\lbrace \t \in (a,b) \mid (u_k(\t))_{k \in \N} \text{ converges to } u(\t) \Big\rbrace \cap \bigcap_{k \in \N} \PPP(u_k,y_k). $$
Note that $\PPP$ is of full measure in $[a,b]$. Let us fix some $k \in \N$. From Inequality~\eqref{ineq}, it holds that
\begin{equation*}
- 2R \sqrt{\eps_k} \leq \dfrac{1}{\J_k \Big( u_k^{(\t,v)}(\cdot,h),y_k \Big) + \J_k (u_k,y_k)} \times \dfrac{\J^2_k \Big( u_k^{(\t,v)}(\cdot,h),y_k \Big) - \J^2_k (u_k,y_k)}{h}, 
\end{equation*}
for all $0 < h \leq b - \t$. Letting $h \to 0^+$, we exactly get from Lemma~\ref{lemneedletraj} and Proposition~\ref{propneedle} that
\begin{multline*}
 \Bigg( \dfrac{(b-\t)^{\alpha-1}}{\Gamma(\alpha)}  \Big( \psi^0_k \partial_2 \varphi (x_k(a),x_k(b)) + \partial_2 g ( x_k(a) , x_k(b) )^\top \times \psi_k \Big) \\
  + \psi^0_k \I^\alpha_{b-} \left[ \dfrac{(b-\cdot)^{\beta-1}}{\Gamma(\beta)}  \partial_1 L (x_k,u_k,\cdot) \right](\t) \Bigg) \cdot (v-u_k(\t)) \\[8pt]
  + \psi^0_k \dfrac{(b-\t)^{\beta-1}}{\Gamma(\beta)}  \Big( L(x_k(\t),v,\t)-L(x_k(\t),u_k(\t),\t) \Big) \leq 2R \sqrt{\eps_k}.
\end{multline*}
Finally, letting $k \to \infty$, we obtain (using in particular the LDC theorem and the fact that $(u_k(\t))_{k \in \N}$ converges to $u(\t)$) that
\begin{multline}\label{eqcrucial1}
 \Bigg( \dfrac{(b-\t)^{\alpha-1}}{\Gamma(\alpha)}  \Big( \psi^0 \partial_2 \varphi (x(a),x(b)) + \partial_2 g ( x(a) , x(b) )^\top \times \psi \Big)  \\
  + \psi^0 \I^\alpha_{b-} \left[ \dfrac{(b-\cdot)^{\beta-1}}{\Gamma(\beta)}  \partial_1 L (x,u,\cdot) \right](\t) \Bigg) \cdot (v-u(\t)) \\[8pt]
  + \psi^0 \dfrac{(b-\t)^{\beta-1}}{\Gamma(\beta)}  \Big( L(x(\t),v,\t)-L(x(\t),u(\t),\t) \Big) \leq 0.
\end{multline}
Note that the above crucial inequality is satisfied for almost all $\t \in [a,b]$ and all $v \in \BB_{\R^n}(0_{\R^n},R)$.

\paragraph{Inroduction of an adjoint vector.} Let us introduce the adjoint vector $p$ defined by
\begin{multline*}
p(t) := \dfrac{(b-t)^{\alpha-1} }{\Gamma(\alpha)} \Big( \psi^0 \partial_2 \varphi (x(a),x(b)) + \partial_2 g ( x(a) , x(b) )^\top \times \psi \Big) \\ + \psi^0 \I^\alpha_{b-} \left[ \dfrac{(b-\cdot)^{\beta-1}}{\Gamma(\beta)} \partial_1 L (x,u,\cdot) \right](t), 
\end{multline*}
for almost every $t \in [a,b]$. In particular it holds that
$$ \I^{1-\alpha}_{b-}[p](t) = \Big( \psi^0 \partial_2 \varphi (x(a),x(b)) + \partial_2 g ( x(a) , x(b) )^\top \times \psi \Big) + \psi^0 \I^1_{b-} \left[ \dfrac{(b-\cdot)^{\beta-1}}{\Gamma(\beta)} \partial_1 L (x,u,\cdot) \right](t), $$
for all $t \in [a,b]$.
We deduce that $p \in \AC^\alpha_{b-}$ with
$$ \Dp [p] (t) = \psi^0 \dfrac{(b-t)^{\beta-1}}{\Gamma(\beta)}  \partial_1 L (x(t),u(t),t), $$
for almost every $t \in [a,b]$. Moreover, in particular from Equality~\eqref{eqcrucial2}, it holds that
\begin{eqnarray*}
- \I^{1-\alpha}_{b-}[p](a) & = & \psi^0 \partial_1 \varphi (x(a),x(b)) + \partial_1 g ( x(a) , x(b) )^\top \times \psi , \\
\I^{1-\alpha}_{b-}[p](b)&  = & \psi^0 \partial_2 \varphi (x(a),x(b)) + \partial_2 g ( x(a) , x(b) )^\top \times \psi .
\end{eqnarray*}
From Inequality~\eqref{eqcrucial1}, it is clear that
\begin{equation}\label{eqmax}
 u(t) \in \argmax_{v \in \BB_{\R^n}(0_{\R^n},R) } \Big\lbrace p(t) \cdot v +\psi^0 \dfrac{(b-t)^{\beta-1}}{\Gamma(\beta)}  L(x(t),v,t) \Big\rbrace,
\end{equation}
for almost every $t \in [a,b]$. One can easily deduce that
$$ p(t) = - \psi^0 \dfrac{ (b-t)^{\beta-1}}{\Gamma(\beta)} \partial_2 L (x(t),u(t),t), $$
for almost every $t \in [a,b]$. 

\paragraph{Normalization.} 
By contradiction, let us assume that $\psi^0 = 0$. In that case, one can easily deduce from the above equalities that $ \partial_1 g(x(a),x(b))^\top \times \psi = \partial_2 g(x(a),x(b)) )^\top \times \psi = 0_{\R^n}$. Since $g$ is assumed to be regular at $(x(a),x(b))$, we deduce that $\psi = 0_{\R^j}$ which raises a contradiction with the nontriviality of the couple $(\psi^0,\psi)$. We deduce that $\psi^0 < 0$. Moreover, since the couple~$(\psi^0,\psi)$ is defined up to a positive multiplicative constant, we now normalize the couple $(\psi^0,\psi)$ such that~$\psi^0 = -1$.

\paragraph{End of the proof.} 
We deduce from the previous paragraphs that $p = \frac{(b-\cdot)^{\beta-1}}{\Gamma(\beta)} \partial_2 L (x,u,\cdot) \in \AC^\alpha_{b-}$ with
$$ \Dp \left[  \dfrac{(b-\cdot)^{\beta-1}}{\Gamma(\beta)}  \partial_2 L (x,u,\cdot) \right] (t) = - \dfrac{(b-t)^{\beta-1}}{\Gamma(\beta)}  \partial_1 L (x(t),u(t),t), $$
for almost every $t \in [a,b]$, which exactly corresponds to the Euler-Lagrange equation stated in Theorem~\ref{thm2} since $u = \CDm[x]$. We also deduce from the equalities on $\I^{1-\alpha}_{b-} [p](a)$ and $\I^{1-\alpha}_{b-} [p](b)$ in a previous paragraph that the transversality conditions given in Theorem~\ref{thm2} are satisfied. Finally, from the maximization condition~\eqref{eqmax}, it is clear that the matrix $\frac{(b-t)^{\beta-1}}{\Gamma(\beta)} \partial^2_{22} L(x(t),u(t),t)$ is positive semi-definite for almost all $t \in [a,b]$, which exactly corresponds to the Legendre condition given in Theorem~\ref{thm2} since $u = \CDm[x]$.

\bibliographystyle{plain}

\end{document}